\documentclass[11pt]{amsart}
\usepackage{graphicx}
\usepackage[english]{babel}
\usepackage{amsthm}
\usepackage{float} 
\usepackage{amsmath}
\usepackage{amssymb}
\usepackage{amsfonts}
\usepackage{color}
\usepackage{algorithm}
\usepackage{algorithmic}
\usepackage{enumerate}
\usepackage{cite}
\usepackage{hyperref}
\usepackage{stmaryrd}
\usepackage{caption}
\usepackage{subcaption}
\usepackage{mathrsfs}
\usepackage{enumitem}
\usepackage{mathrsfs}  
\usepackage{tikz-cd} 
\usepackage{pdfsync}
\usepackage{geometry}
\usepackage{esint}
\usepackage{url}

\numberwithin{equation}{section}

\allowdisplaybreaks[3]

\newtheorem{thm}{Theorem}[section]
\newtheorem{lemma}[thm]{Lemma}

\newtheorem{cor}[thm]{Corollary}
\newtheorem{prop}[thm]{Proposition}
\theoremstyle{definition}
\newtheorem{definition}[thm]{Definition}
\theoremstyle{remark}
\newtheorem{remark}[thm]{Remark}

\usepackage{array}

\usepackage{blindtext}
\DeclareMathOperator{\curl}{curl}
\newcommand{\dofs}{dofs\,}
\newcommand{\dof}{dof}
\newcommand{\comment}[1]{}
\newcommand{\trans}{\tau}
\newcommand{\cV}{\mathcal{V}}
\newcommand{\bV}{\mathbb{V}}
\newcommand{\bR}{\mathbb{R}}
\newcommand{\cS}{\mathcal{S}}
\newcommand{\cL}{\mathcal{L}}
\newcommand{\cR}{\mathcal{R}}
\newcommand{\cP}{\mathcal{P}}

\newcommand{\calT}{\mathcal{T}}
\newcommand{\THWFT}{\calT^{wf}_h}
\newcommand{\WFT}{T^{wf}}
\newcommand{\Div}{{\rm div}\,}
\newcommand{\DivF}{{\rm div}_\F}
\newcommand{\rot}{{\rm rot}}
\newcommand{\bn}{n}
\newcommand{\bt}{t}
\newcommand{\bs}{s}
\newcommand{\bcurl}{{\rm curl}\,}
\newcommand{\Grad}{{\rm grad}}
\newcommand{\mskw}{{\rm mskw}\,}
\newcommand{\inc}{{\rm \bcurl \Xi^{-1} \bcurl}}
\newcommand{\Inc}{{\rm inc}\,}
\newcommand{\tr}{{\rm tr}}
\newcommand{\vskw}{{\rm vskw}}
\newcommand{\skw}{{\rm skw}}
\newcommand{\sym}{{\rm sym}}
\newcommand{\bM}{\mathbb{M}}
\newcommand{\bK}{\mathbb{K}}
\newcommand{\bS}{\mathbb{S}}
\newcommand{\Fct}{F^{\ensuremath\mathop{\mathrm{ct}\,}}}
\newcommand{\Rig}{\mathsf{R}}
\newcommand{\Lag}{\mathsf{X}}
\newcommand{\jmp}[1]{ [\![ {#1}  ]\!] }
\newcommand{\FF}{{\scriptscriptstyle{F\!F}}}
\newcommand{\F}{{\scriptscriptstyle{F}}}
\newcommand{\nF}{{n\F}}
\newcommand{\Fn}{{\F n}}
\newcommand{\incF}{{\rm inc}_\F}
\newcommand{\airy}{{\rm airy}_\F}
\newcommand{\U}{{\scriptscriptstyle{U}}}

\newcommand{\dofcnt}[1]{{\tiny{\text{\ensuremath{{#1}~{\rm{dofs}}}}}}}

\newcommand{\rev}[1]{\textcolor{black}{#1}}
\newcommand{\bbR}{\mathbb{R}}

\usepackage{bm}
\title{Discrete Elasticity Exact Sequences on Worsey-Farin splits}
\author[S. Gong]{Sining Gong}
\address{Division of Applied Mathematics, Brown University, Providence, RI 02912 }
\email{sining\_gong@alumni.brown.edu}

\author[J. Gopalakrishnan]{Jay Gopalakrishnan}
\address{Portland State University (MTH),
  PO Box 751, Portland, OR 97207, USA}
\email{gjay@pdx.edu}

\author[J. Guzm\'an]{Johnny Guzm\'an}
\address{Division of Applied Mathematics, Brown University, Providence, RI 02912 }
\email{johnny\_guzman@brown.edu}           

\author[M. Neilan]{Michael Neilan}   
\address{Department of Mathematics, University of Pittsburgh, Pittsburgh, PA 15260}
\email{neilan@pitt.edu}

\thanks{
This work was supported in part by NSF grants DMS--1913083, DMS--2245077,  and DMS--2011733.}

\begin{document}

\maketitle

\begin{abstract} 
  We construct conforming finite element elasticity complexes on Worsey-Farin splits 
  in three dimensions. Spaces for displacement, strain, stress, and the load
  are connected in 
  the elasticity complex through the differential operators representing 
  deformation, incompatibility, and divergence.
  For each of these component spaces,
  a corresponding finite element space on Worsey-Farin meshes is exhibited.
  Unisolvent degrees of freedom are developed for these finite elements,
  which also yields  commuting (cochain) projections on smooth functions.  
A distinctive feature of the spaces in these complexes is the lack of extrinsic supersmoothness
at subsimplices of the mesh.  Notably,
the complex yields 
the first (strongly) symmetric stress \rev{finite} element with no vertex or edge degrees of freedom in three dimensions. 
  Moreover, the lowest order stress space uses only piecewise linear functions 
  which is the lowest feasible polynomial degree for the stress space. 
\end{abstract}

\thispagestyle{empty}

\section{Introduction}

The elasticity complex, also known as the Kr\"oner complex,
can be derived from simpler complexes by an algebraic technique called
the \rev{Bernstein-Gelfand-Gelfand (BGG) resolution \cite{arnold2021complexes, CapSlovaSouce01,eastwood2000complex, chen2022complexes}}. The  utility of the BGG construction  in developing and understanding stress elements for elasticity is now well appreciated~\cite{FE2006}.
However even with this machinery, the construction of conforming, inf-sup stable stress  elements on simplicial 
meshes is still a notoriously challenging task~\cite{boffi2008finite}.
It was not until 2002 that the first conforming elasticity elements
were successfully constructed on two-dimensional triangular meshes by
Arnold and Winther~\cite{arnold2002mixed}. There, they argued {\rev{that
 degrees of freedom (``\dofs'') on  vertices}} are necessary when using 
polynomial approximations on triangular elements. They in fact constructed an
entire discrete elasticity complex and showed how the last two spaces
there are relevant for discretizing the Hellinger-Reissner principle in
elasticity.


Following the creation of the first two-dimensional (2D) conforming
elasticity elements, the first three-dimensional (3D) elasticity
elements were constructed in \cite{adams2005mixed,arnold2008finite},
which paved the way for many other similar elements, as demonstrated
in \cite{hu2015family}.  A natural question that arose was whether
these elements could be seen as part of an entire discrete elasticity
complex, similar to what was done in 2D. Although the work
in~\cite{arnold2008finite} laid the foundation, the task of extending
it to 3D was bogged down by complications.
This is despite the clearly understood 
 BGG procedure to arrive at an elasticity complex 
 of smooth function spaces,
\begin{equation}\label{eq:smooth-complex}
\begin{tikzcd}
  0 \arrow{r}
  &\Rig \arrow{r}{\subset}
  &C^{\infty}\otimes \bV \arrow{r}{\varepsilon}
  & C^{\infty}\otimes \mathbb{S}  \arrow{r}{\text{inc}}
  & C^{\infty}\otimes \mathbb{S}\arrow{r}{\text{div}}
  &C^{\infty}\otimes {\bV}  \arrow{r}&0.
 \end{tikzcd}
\end{equation}
Here and throughout, $\bV = \bR^3$, $\bM = \bR^{3 \times 3}$,
$\Rig=\{ a+b \times x: a, b \in \mathbb{R}^3\}$ denotes rigid
displacements, $\text{inc} =\curl \circ \trans \circ \curl$ with
$\trans$ denoting the transpose, curl and divergence operators are
applied row by row on matrix fields, $\mathbb{S} = \text{sym} (\bM)$,
and $\varepsilon = \text{\sym} \circ \text{grad}$ denotes the
deformation operator.  The complex~\eqref{eq:smooth-complex} is exact
on a 3D contractible domain. We assume throughout that our domain
$\Omega$ is contractible. To give an indication of the aforementioned
complications, first note that the techniques leading up to those
summarized in~\cite{arnold2021complexes} showed how the BGG
construction can be extended beyond smooth complexes
like~\eqref{eq:smooth-complex}.  For example, applying the BGG
procedure to de Rham complexes of Sobolev spaces
$H^s \equiv H^s(\Omega)$, the authors of~\cite{arnold2021complexes} arrived at the
following elasticity complex of Sobolev spaces:
\begin{equation}
  \label{eq:6}
  \begin{tikzcd}[]
    \Rig \arrow{r}{\subset}
    & 
    H^s \otimes \bV \arrow{r}{\varepsilon}
    &
    H^{s-1} \otimes \bS \arrow{r}{\text{inc}}
    &
    H^{s-3} \otimes \bS \arrow{r}{\text{div}}
    &
    H^{s-4} \otimes \bV \arrow{r}
    & 0. 
  \end{tikzcd}
\end{equation}
However, one of the problems in constructing finite element
subcomplexes of~\eqref{eq:6} is the increase of four orders of
smoothness from the last space ($H^{s-4}$) to the first space
($H^s$). A search for finite element subcomplexes of elasticity complexes with
different Sobolev spaces seemed to hold more
promise~\cite{arnold2008finite}.

It was not until 2020 that the first 3D discrete elasticity subcomplex
was established in~\cite{christiansen2020discrete}.  To understand
that work, it is useful to look at it from the perspective of applying
the BGG procedure to a different sequence of Sobolev spaces. Starting
with a Stokes complex,  lining up another de Rham complex with
different gradations of smoothness, and applying the BGG procedure,
one gets
\begin{equation}
  \label{eq:the-complex}
  \begin{tikzcd}[ampersand replacement=\&]
    \rev{\Rig} \arrow{r}{\subset}
    \&
    H^2\otimes {\mathbb{V}} \arrow{r}{{\varepsilon}}
    \&
    \rev{H^1}({\operatorname{inc}})  \arrow{r}{{\operatorname{inc}}}
    \&
    H(\mathop{\operatorname{div}}, {\mathbb{S}}) \arrow{r}{\mathop{\operatorname{div}}}
    \&
    L^2\otimes {\mathbb{V}} \arrow{r}
    \&
    0,
  \end{tikzcd}  
\end{equation}
\rev{where $H^1(\operatorname{inc}) = \{ g \in H^1 \otimes \mathbb S :
  \operatorname{inc} g \in L^2 \otimes \mathbb S\}$.
The proof of exactness  of~\eqref{eq:the-complex}}
is described in more detail in \cite[p.~38--40]{MFO22}. 
The key innovation in~\cite{christiansen2020discrete} was the
construction of two sequences of finite element spaces on
which this BGG argument can
be replicated at the discrete level, resulting in a fully discrete
subcomplex of \eqref{eq:the-complex}. These new finite element sequences were
inspired by the ``smoother'' discrete de Rham complexes (smoother than
the classical N\'ed\'elec spaces \cite{nedelec1980mixed}) recently
being produced in a variety of settings \cite{guzman2022exact, guzman2020exact,
  FuGuzman, christiansen2022finite,
  christiansen2018generalized}. Specifically, the 3D discrete 
  \rev{sub-complex of \eqref{eq:the-complex}}
in~\cite{christiansen2020discrete} was built on meshes of Alfeld
splits, a particular type of macro element.  Soon after the results
of~\cite{christiansen2020discrete} were publicized, Chen and
Huang~\cite{chen2022finite} obtained another 3D discrete elasticity
sequence on general triangulations.  \rev{There,
  they produced a finite element subcomplex of another exact sequence obtained from~\eqref{eq:the-complex} by replacing $H^2 \otimes \mathbb V$ and $H^1(\operatorname{inc})$ with
  $H^1 \otimes \mathbb V$ and $H(\operatorname{inc}) = \{ g \in L^2 \otimes \mathbb S : \operatorname{inc} g \in L^2 \otimes \mathbb S\}$, respectively.
A related work is \cite{chen2022complexes}, 
where several finite element elasticity complexes are constructed 
with various smoothness.}
The BGG construction was also applied to obtain discrete tensor product
spaces in~\cite{bonizzoni2023discrete}.


In this paper, we apply the methodology presented
in~\cite{christiansen2020discrete} to construct a new discrete
elasticity sequence on Worsey-Farin splits
\cite{worsey1987n}.  One of the expected benefits of using
triangulations of macroelements is the potential reduction of polynomial degree
and the  potential escape from the unavoidability~\cite{arnold2008finite} of 
vertex degrees of freedom in stress elements.
We will see that Worsey-Farin splits offer
structures where these benefits can be reaped easier than on Alfeld
splits.
Unlike Afleld splits, which divide each 
tetrahedron into four sub-tetrahedra, Worsey-Farin triangulations split each tetrahedron into twelve sub-tetrahedra. Using the
Worsey-Farin split, 
we are able to reduce the polynomial degree.
Previous works have used either quadratics \cite{christiansen2020discrete} or quartics \cite{chen2022finite} 
as the lowest polynomial order for the stress spaces. However, our approach results in stress spaces that are 
{\em piecewise linear stress elements},
which is the lowest possible polynomial degree.
Furthermore, it results in the first 3D symmetric conforming
stress  \rev{finite} element  {\em without edge and vertex \dofs.} 
This is comparable to the 2D elasticity element without vertex \dofs\ constructed in \cite{arnold1984family, guzman2014symmetric}.
\rev{(Note that discrete symmetric stress spaces
  without vertex or edge \dofs\ have also been constructed in \cite{dassi2020three} using a virtual element methodology.) One other notable feature of our Worsey-Farin elements is the {\em lack of extrinsic supersmoothness}, i.e., our \dofs\ do not impose more smoothness than what is intrinsic to Worsey-Farin splits. In contrast,  the \dofs\ of the  discrete elements in \cite{christiansen2020discrete} on Alfeld splits impose additional extrinsic supersmoothness.}

Although we have the framework in \cite{christiansen2020discrete} to
guide the construction of the discrete complex on Worsey-Farin splits,
as we shall see, we face significant new difficulties peculiar to
Worsey-Farin splits.
The most troublesome of these arises in the construction of \dofs\  and
corresponding commuting projections.  Unlike Alfeld splits,
Worsey-Farin triangulations induce a Clough-Tocher split on each face
of the original, unrefined triangulation.  As a result, discrete 2D
elasticity complexes with respect to Clough-Tocher splits play an
essential role in our construction and proofs. These 2D complexes are
more complicated than their analogues on Alfeld splits (where the
faces are not split). The resulting difficulties are most evident
in the design of \dofs\  for the space before the stress space (named $U_r^1$ later) in the complex, as we shall see 
in Lemma~\ref{lem:dofu1}.


The paper is organized as follows. In the next section, we present the 
main framework to construct the elasticity sequence, define the construction of Worsey-Farin splits, and state 
the definitions and notation used throughout the paper. Section \ref{sec:2dela} gives useful de Rham sequences and 
elasticity sequences on Clough-Tocher splits. Section \ref{sec:localcomplex} gives the construction of the discrete elasticity sequence locally on Worsey-Farin splits with the dimensions of each spaces involved. This leads to our main contribution in Section \ref{sec:LocalDOFs} where we present the degrees of freedom of the discrete spaces in the elasticity sequence with commuting projections.
We finish the paper with the analogous global discrete elasticity sequence in Section \ref{sec:global} and state some conclusions and future directions in Section \ref{sec:Conclude}.

\section{Preliminaries} \label{sec:Pre}

\subsection{A derived complex from two complexes}
Our strategy to obtain an elasticity sequence uses the framework in \cite{arnold2021complexes}
and utilizes two auxiliary de Rham complexes. In particular, we will use a simplified version of their results found in \cite{christiansen2020discrete}. 

Suppose $A_i, B_i$ are Banach spaces, $r_i:A_i \rightarrow A_{i+1}$, $t_i: B_i \rightarrow B_{i+1}$, and $s_i: B_i \rightarrow A_{i+1}$ are bounded linear operators such that the following diagram commutes:
\begin{equation}\label{eqn:seqpattern}
    \begin{tikzcd}
A_0 \arrow{r}{r_0} & A_1 \arrow{r}{r_1} & A_2 \arrow{r}{r_2} & A_3 \\ 
B_0 \arrow{r}{t_0} \arrow[swap]{ur}{s_0} &  B_1 \arrow{r}{t_1} \arrow[swap]{ur}{s_1} & B_2 \arrow{r}{t_2} \arrow[swap]{ur}{s_2}& B_3
\end{tikzcd}
\end{equation}
The following recipe for a derived complex, borrowed from  \cite[Proposition 2.3]{christiansen2020discrete},  guides the gathering of ingredients for our
construction of the elasticity complex on Worsey-Farin splits.

\begin{prop}\label{prop:exactpattern}
    Suppose $s_1:B_1\to A_2$ is a bijection.
\begin{enumerate}
\item If $A_i$ and $B_i$ are exact sequences and the diagram \eqref{eqn:seqpattern} commutes, then the following is an exact sequence:
  \begin{equation}
    \label{eq:1}
  \begin{bmatrix}
        A_0 \\
        B_0
    \end{bmatrix} \xrightarrow{[\begin{smallmatrix}
        r_0 & s_0
    \end{smallmatrix}]}  A_1 \xrightarrow{t_1 \circ s_1^{-1} \circ r_1 } B_2 \xrightarrow{\left[\begin{smallmatrix}
        s_2 \\ t_2
    \end{smallmatrix}\right]}
    \begin{bmatrix}
        A_3 \\ B_3
    \end{bmatrix}.
  \end{equation}
Here the operators $[r_0\ s_0]:\begin{bmatrix}A_0\\B_0\end{bmatrix}\to A_1$
and $\begin{bmatrix}s_2\\ t_2\end{bmatrix}:B_2\to \begin{bmatrix}A_3\\B_3\end{bmatrix}$ are defined, respectively, as
\[
[r_0\ z_0]\begin{bmatrix}a\\b\end{bmatrix} = r_0 a +z_0 b,\qquad
\begin{bmatrix}s_2\\t_2\end{bmatrix}
b = \begin{bmatrix}
s_2 b\\
t_2 b
\end{bmatrix}.
\]
\item For the surjectivity of the last map in~\eqref{eq:1}, namely $\left[\begin{smallmatrix}
        s_2 \\ t_2
                                                                          \end{smallmatrix}\right]$,
                                                                        it is sufficient that $r_2$ and $t_2$ are surjective, $t_1 \circ t_2=0$, and $s_2t_1 = r_2s_1$. 
    \end{enumerate}
\end{prop}


\subsection{Construction of Worsey-Farin Splits}\label{subsec:WFconstruct}

For a set of simplices $\cS$, we use $\Delta_s(\cS)$ to denote the set of 
$s$-dimensional simplices ($s$-simplices for short) in $\cS$.  
If $\cS$ is a simplicial triangulation of a domain $D$ with boundary, then $\Delta_s^I(\cS)$ denotes the 
subset of $\Delta_s(\cS)$ that {does} not belong to the boundary
of the domain.  If $S$ is a simplex, then we use the convention $\Delta_s(S) = \Delta_s(\{S\})$. For a non-negative integer $r$, we use $\mathcal{P}_r(S)$ to denote the space of \rev{polynomials} of degree $\leq r$ on $S$,
 and we define
 \begin{align*}
\cP_r(\cS) = \prod_{S\in \cS} \cP_r(S), \quad L^2_0(D):=\{q \in L^2(D): \int_D q~ dx = 0\}. 
 \end{align*}

Let $\Omega\subset \bbR^3$ be a contractible 
polyhedral domain, and let $\{\calT_h\}$ be \rev{a} family of shape-regular
and simplicial triangulations of $\Omega$.
The Worsey-Farin refinement of $\calT_h$, denoted by
$\mathcal{T}^{wf}_h$, is obtained by
splitting each $T \in \mathcal{T}_h$ by the following two steps (cf.~\cite[Section 2]{guzman2022exact} \rev{and Figure \ref{fig:three graphs}}):
\begin{enumerate}
\item  Connect the incenter $z_T$ of $T$
 to its (four) vertices. 
\item For each face $F$ of $T$ choose $m_\F \in  {\rm int}(F)$.  We then connect $m_\F$ to the three vertices of $F$ and to the incenter $z_T$.
\end{enumerate}
Note that the first step is an Alfeld-type refinement of $T$ with respect to the incenter \cite{christiansen2020discrete}.
We denote the local mesh of the Alfeld-type refinement by $T^a$, which consists of four tetrahedra.
The choice of the point $m_\F$ in the second step needs to follow specific rules: for each interior face $F = \overline{T_1} \cap \overline{T_2}$ with $T_1$,
 $T_2 \in \calT_h$, let $m_\F = L \cap F$ where $L = [{z_{T_1},z_{T_2}}]$, the line segment connecting the incenters of $T_1$ and $T_2$; for a boundary face $F$ with $F = \overline{T} \cap \partial \Omega$ with $T\in \calT_h$, let $m_\F$ be the barycenter of $F$. The fact that such a $m_\F$ exists is established in \cite[Lemma  16.24]{lai2007spline}.

For $T\in \calT_h$, we denote by $\WFT$ the local
Worsey-Farin mesh induced by the global refinement $\calT_h^{wf}$, i.e.,
\[
\WFT = \{K\in \calT_h^{wf}:\ \bar K\subset \bar T\}.
\]
For any face $F \in \Delta_2(\mathcal{T}_h)$, the refinement $\calT_h^{wf}$ induces a Clough-Tocher triangulation of $F$, i.e.,
a two-dimensional triangulation consisting of three triangles, each having
the common vertex $m_\F$; we denote this set of three triangles by $F^{ct}$; \rev{see Figure \ref{fig:Fct}}.
We then define
\[
\mathcal{E}(\mathcal{T}_h^{wf}) = \{e \in \Delta_1^I(F^{ct}): \text{ for all } F \in \Delta_2^I(\mathcal{T}_h) \}
\] 
to be the set of all interior edges of the Clough-Tocher refinements
in the global mesh.

For a tetrahedron $T \in \calT_h$ and face $F \in \Delta_2(T)$, we denote by $\bn_\F := \bn|_\F$ the outward unit normal of $\partial T$ restricted to $F$. Consider the triangulation $\Fct$ of $F$ with three triangles labeled as $Q_i$, $i =1,2,3$. Let $e = \partial Q_1 \cap \partial Q_2$ and $\bt_e$ be the unit vector tangent to $e$ pointing away from $m_\F$. Then the jump of $p \in \cP_r(\WFT)$ across $e$ is defined as
\[
\jmp{p}_{e} = (p|_{Q_1}-p|_{Q_2})\bs_e,
\]
where $\bs_e = \bn_\F \times \bt_e$ is a unit vector orthogonal to $\bt_e$ and $\bn_\F$. In addition, let $f$ be the  internal face of $\WFT$ that has $e$ as an edge. Now let $\bn_f$ be a unit-normal to $f$ and set $\bt_s = \bn_f \times \bt_e$ to be a tangential unit vector on the internal face $f$. 

Let $T_1$ and $T_2$ be two adjacent tetrahedra in $\calT_h$ that share a face $F$, and let $Q_i$, $i =1,2,3$ denote three triangles in the set $\Fct$. Let $e = \partial Q_1 \cap \partial Q_2$, and for a piecewise smooth function defined on 
$T_1 \cup T_2$, we define
  \begin{equation}\label{eqn:thetaEDef}
 \theta_e(p) = p|_{\partial T_1 \cap Q_1}-p|_{\partial T_1 \cap Q_2}+p|_{\partial T_2 \cap Q_2}-p|_{\partial T_2 \cap Q_1},~~~~~~ \text{on~} e.
 \end{equation}
 Note that $\theta_e(p)=0$ if and only if $\jmp{p|_{T_1}}_{e}=\jmp{p|_{T_2}}_{e}$.
\begin{figure}
      \centering
      \begin{subfigure}[b]{0.3\textwidth}
          \centering
          \includegraphics[width=\textwidth]{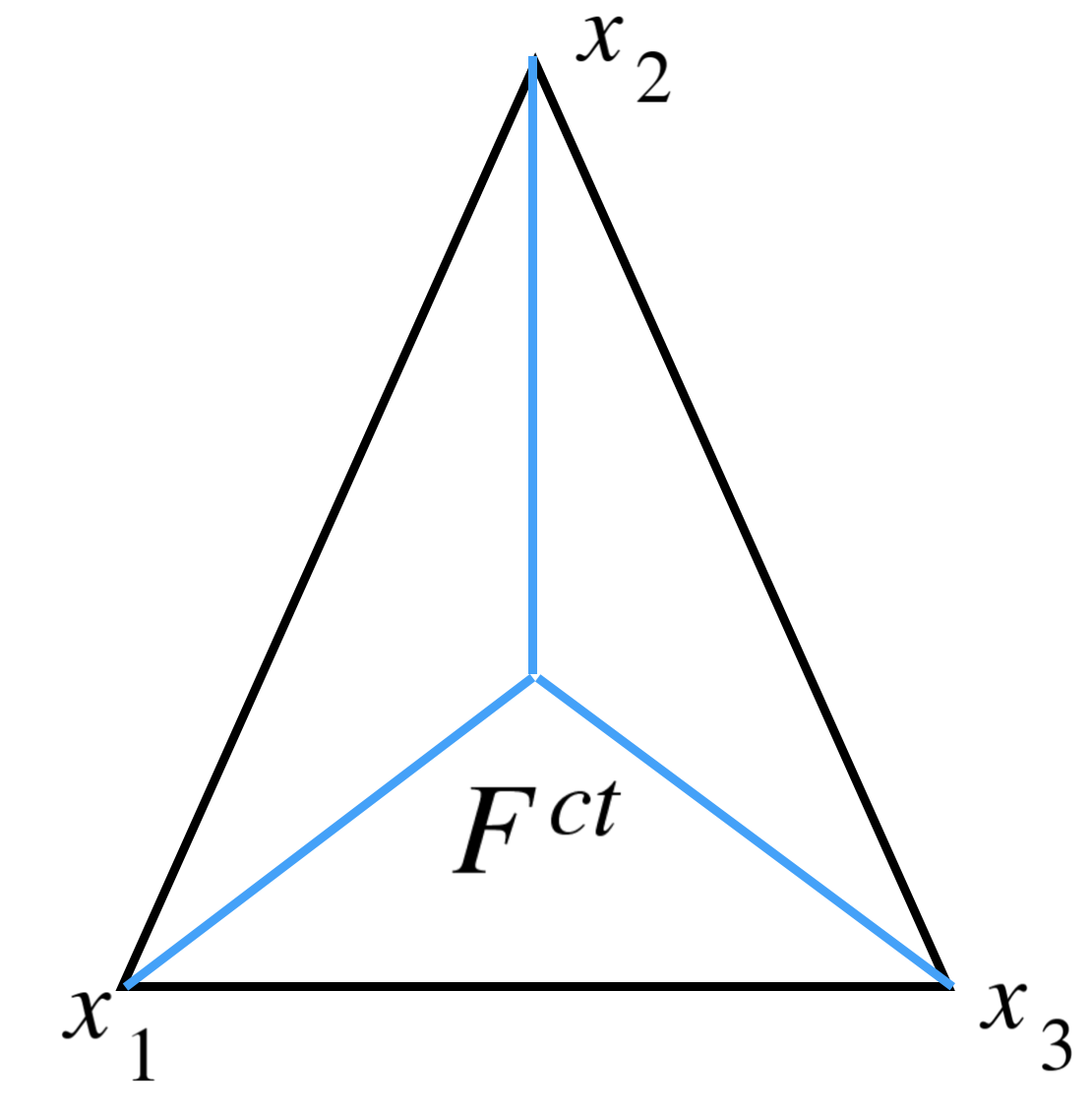}
          \caption{A representation of $F^{ct}$ and $\Delta_1^I(F^{ct})$ (indicated in blue).}
          \label{fig:Fct}
      \end{subfigure}
      \hfill
      \begin{subfigure}[b]{0.3\textwidth}
          \centering
          \includegraphics[width=\textwidth]{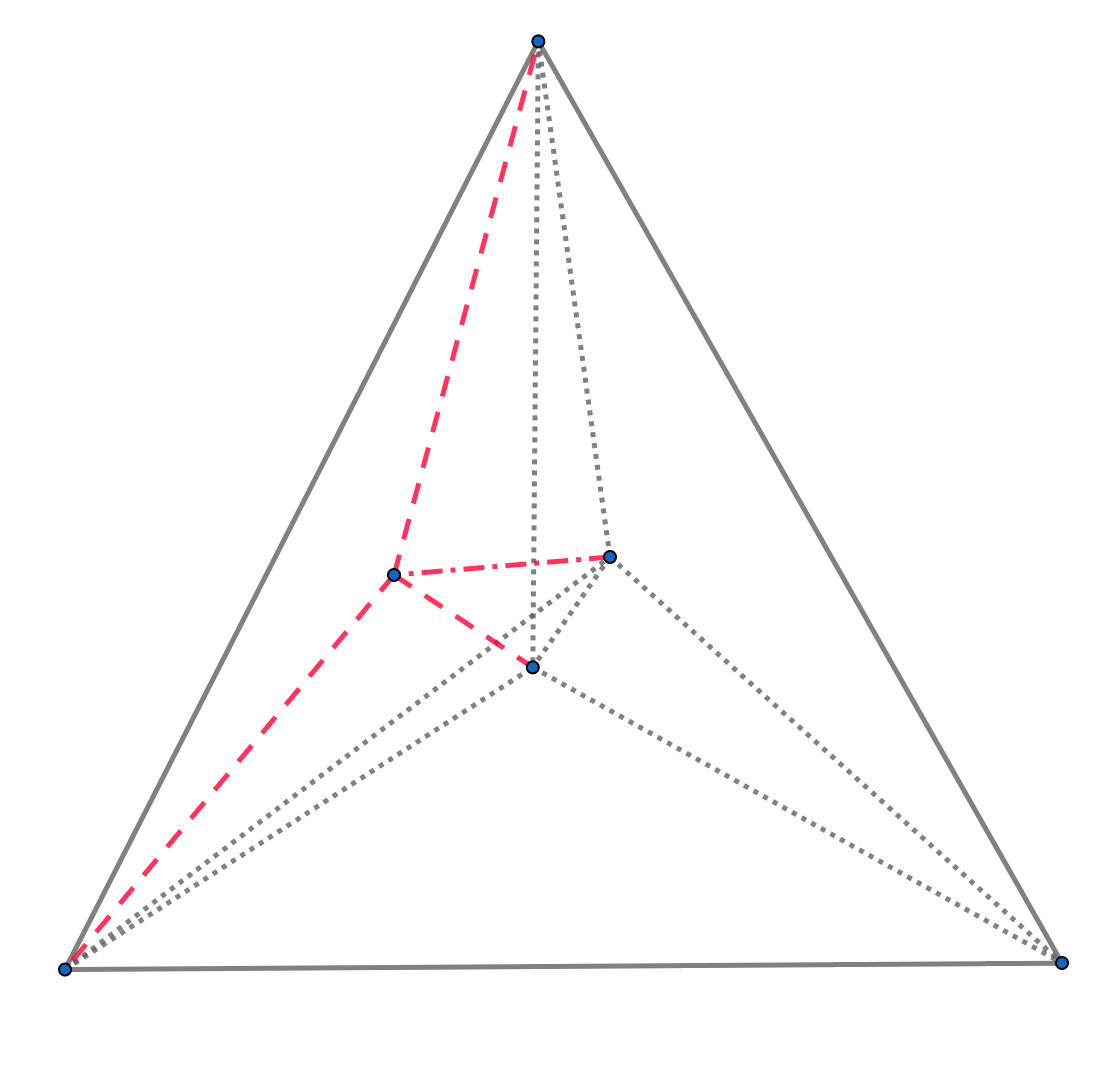}
          \caption{\rev{Alfeld refinement and Worsey-Farin refinement (local) indicated in red}}
          \label{fig:WF-show}
      \end{subfigure}
      \hfill
      \begin{subfigure}[b]{0.3\textwidth}
          \centering
          \includegraphics[width=\textwidth]{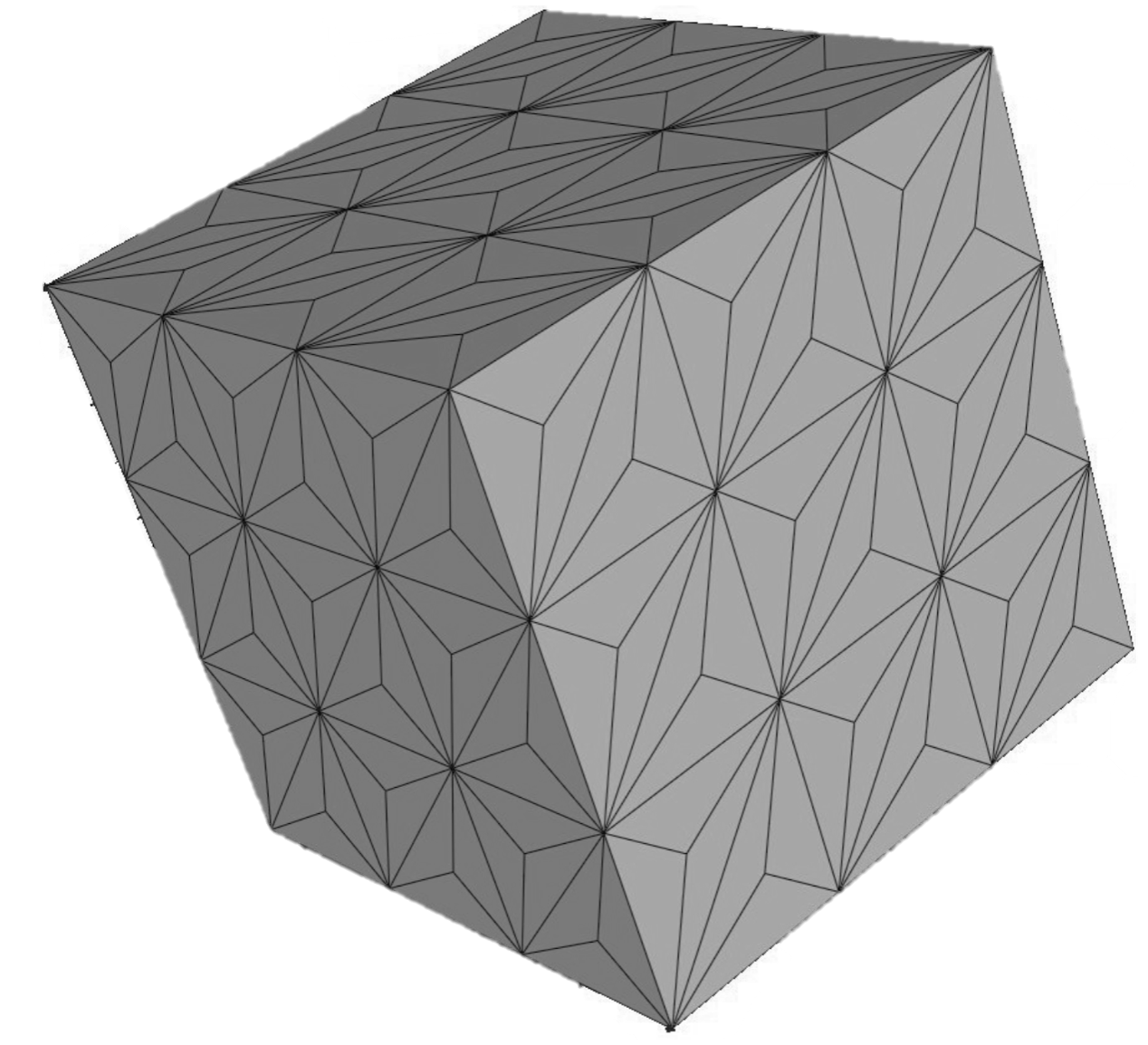}
          \caption{Worsey-Farin refinement (global)}
          \label{fig:WF2}
      \end{subfigure}
         \caption{The Worsey-Farin Splits}
         \label{fig:three graphs}
 \end{figure}

\subsection{Differential identities involving matrix and vector fields}\label{subsec:Matrix}
We adopt the notation used in \cite{christiansen2020discrete}. Let $F \in \Delta_2(T)$, and recall $\bn_\F$ is the unit normal vector
pointing out of $T$. Fix two tangent vectors $\bt_1, \bt_2$ 
such that the ordered set 
$(b_1, b_2, b_3) = (\bt_1, \bt_2, \bn_\F)$ is an
orthonormal right-handed basis of $\mathbb{R}^3$. Any matrix field
$u: T \to \mathbb{R}^{3 \times 3}$ can be written as
$ \sum_{i, j=1}^3 u_{ij}b_i b_j'$ with scalar components
$u_{ij}: T \to \bR$.  Let $ u_{nn} = \bn_\F' u \bn_\F$ and
$ \tr_\F u = \sum_{i=1}^2 \bt_i' u \bt_i.$ 
With $s\in \bbR^3$, let
\begin{equation}
  \label{eq:11}
  u_{\FF} =\sum_{i, j = 1}^2 u_{ij}\bt_i \bt_j',
  \qquad
  u_{F s}= \sum_{i=1}^2 ( s'u\bt_i) \bt_i',
  \qquad
  u_{sF}= \sum_{i=1}^2 (\bt_i' u s) \bt_i,. 
\end{equation}
Equivalently,
$u_{\FF} = Qu Q$, $u_{F s} =s' uQ,$ and $u_{s F} = Q u s$, where
$P = \bn_\F\bn_\F'$ and $Q = I-P$.
Next, for  scalar-valued (component) functions
$\phi, w_i, q_i$ and $u_{ij}$, we write the standard surface operators as
\begin{align*}
  & \Grad_\F \phi  = (\partial_{\bt_1} \phi)\bt_1 +(\partial_{\bt_2}\phi)  \bt_2,
  &
    \Grad_\F (w_1\bt_1 + w_2 \bt_2)  = \bt_1 (\Grad_\F w_1)' + \bt_2(\Grad_\F w_2)',&
  \\
  &\rot_\F \phi  =(\partial_{\bt_2}\phi) \bt_1 - (\partial_{\bt_1} \phi) \bt_2,
  &
    \rot_\F (q_1\bt_1' + q_2 \bt_2') = \bt_1 (\rot_\F q_1)' + \bt_2(\rot_\F q_2)',&
  \\
  & \curl_\F (w_1 \bt_1+w_2 \bt_2)  = \partial_{\bt_1} w_2 -\partial_{\bt_2} w_1, 
  &   \curl_\F u_{\FF}
    = \bt_1' \,\curl_\F (u_{F \bt_1})' + \bt_2'\,\curl_\F (u_{F \bt_2})', &  \\
  & \DivF (w_1 \bt_1+w_2 \bt_2)  = \partial_{\bt_1} w_1 + \partial_{\bt_2} w_2, 
  &   \DivF u_{\FF}
    = \bt_1' \,\DivF (u_{F \bt_1})' + \bt_2'\,\DivF (u_{F \bt_2})'. & 
\end{align*}
\rev{These operators are defined such that they are consistent with the conventions in \cite{christiansen2020discrete}.
In particular, we define $\rot_\F$, such that the resulting operator $\airy$ mimics the three-dimensional operator, $\Inc$.}
For a vector function $v$, denote 
$v_\F = Q v = \bn_\F \times (v \times \bn_\F)$. It is easy to see that
\begin{equation}
  \label{eq:10}
  \begin{aligned}
  \bn_\F \cdot \bcurl v &= \bcurl_\F v_\F, \quad
  &&(\Grad \, v)_{\FF} = \Grad_\F v_\F,\\
  \quad \bn_\F \times \rot_\F \phi &= \Grad_\F \phi, \quad &&\Div v_\F = \DivF v_\F.
  \end{aligned}
\end{equation}

\begin{definition}\label{def:Vperp}
For a tangential vector function $v$ on the face $F \in \Delta_2(T)$,
write $v = \sum\limits_{i=1}^2 v_i \bt_i$ with $v_i = v \cdot \bt_i$. 
    We define the orthogonal complement of $v$ as
    \[
    v^\perp = v_2 \bt_1 - v_1 \bt_2. 
    \]
\end{definition}    
Using this definition and the standard surface operators introduced above, it is easy to see the following identities:
\begin{equation}\label{iden:Vperp}
    \DivF v^\perp = \curl_\F v, \quad v^\perp \cdot \bt_e = v \cdot \bs_e,\quad v^{\perp} = v \times \bn_\F.
\end{equation}

     We also define the space of rigid body displacements within $\mathbb{R}^3$ and the face $F$: 
    \begin{alignat}{1} 
        \label{eqn:rigid1}
        \Rig & =\{a+b \times x: a, b \in \mathbb{R}^3\} \\
        \label{eqn:rigid2}
        \Rig(F)& = \{a \bt_1 + b \bt_2 + c((x \cdot \bt_1)\bt_2 - (x \cdot \bt_2)\bt_1): a, b, c \in \mathbb{R}\}.
    \end{alignat}


\begin{definition}\label{def:maps}\
 Set $\bV = \bR^3$, 
 and $\bM_{k \times k} = \bbR^{k\times k}$. 
    \begin{enumerate}
        \item The skew-symmetric operator $\skw: \bM_{k \times k} \rightarrow \bM_{k \times k}$ and the symmetric operator $\sym: \bM_{k \times k} \rightarrow \bM_{k \times k}$ are defined as follows: for any $M \in \bM_{k \times k}$, 
        \[
        \skw(M) = \frac{1}{2}(M-M'); \quad \sym(M) = \frac{1}{2}(M+M').
        \]
        Denote the range of $\skw$ and $\sym$ as $\bK_k = \skw(\bM_{k \times k})$ and $\bS_k = \sym(\bM_{k \times k})$, respectively.  
\item Define the operator $\Xi : \bM_{3 \times 3} \rightarrow \bM_{3 \times 3}$ by $\Xi M = M'-{\rm tr} (M) \mathbb{I}$, where $\mathbb{I}$ is the $3\times 3$
identity matrix.\medskip
\item The three-dimensional symmetric gradient
and incompatibility operators
 are given, respectively, by:
\[
 \varepsilon = \sym \, \Grad, \quad \Inc = \curl (\curl)'. 
\]
\item The operators $\mskw: \mathbb{V} \rightarrow \mathbb{K}_3$ 
and $\vskw:\bM_{3\times 3}\to \mathbb{V}$ are given by 
    \[\mskw \begin{pmatrix}
        v_1 \\
        v_2 \\
        v_3
    \end{pmatrix} = \begin{pmatrix}
        0 & -v_3 & v_2 \\
        v_3 & 0 & -v_1 \\
        -v_2 & v_1 & 0
    \end{pmatrix},\qquad \vskw := \mskw^{-1} \circ \skw.
    \]
\item The two-dimensional surface differential operators on a face $F$
are given by
\[
\varepsilon_\F = \sym \, \Grad_\F, \quad 
\airy = \rot_\F(\rot_\F)', \quad 
\incF := \curl_F(\curl_F)'.
\]

\item    The two-dimensional skew operator
defined on either a scalar or matrix-valued function
is defined, respectively, as
    \[
    {\rm skew} \, u = 
    \begin{bmatrix}
        0 & u \\
        -u & 0 
    \end{bmatrix}; \quad 
    {\rm skew} \, \begin{bmatrix}
        u_{11} & u_{12} \\
        u_{21} & u_{22} 
    \end{bmatrix} = u_{21} - u_{12}.
    \]
\item    The transpose operator $\tau$ is defined as: $\tau \, u = u'$.
    \end{enumerate}
\end{definition}


It is simple to see that $\Xi$ is invertible
with $\Xi^{-1} M = M'-\frac12 {\rm tr}(M)\mathbb{I}$.
Furthermore, the following  identities hold:
\begin{subequations}
\begin{alignat}{1}
\label{eqn:iden1}
   & \Div \Xi  = 2\vskw \, \bcurl, \\
\label{eqn:iden2}
  &   \Xi \Grad = - \bcurl \, \mskw, \\
\label{eqn:iden3}
& \inc \mskw = - \bcurl \Xi^{-1} \Xi \Grad = -\bcurl \Grad = 0, \\
\label{eqn:iden4}
& 2\, \vskw \, \inc = \Div \Xi \Xi^{-1} \bcurl = \Div \bcurl = 0, \\
\label{eqn:iden5}
&  \text{tr}(\bcurl \sym)=0, \quad \inc \sym = \curl(\curl \sym)' = \Inc \sym.
\end{alignat}
\end{subequations}
On a two-dimensional face $F$, there also holds
\begin{subequations}
    \begin{alignat}{1} 
    \label{2didenela1}
    &\DivF \, \airy = \DivF \, \rot_\F\, \tau\,(\rot_\F) = 0, \\
    \label{2didenela2}
    &\incF \, \sym = \incF, \quad 
    \incF \, \varepsilon_\F = \curl_\F \, \tau\, \curl_\F \Grad_\F = 0, \\
    \label{2didenela3}
    &\curl_\F \, {\rm skew}  = \tau \, \Grad_\F.
\end{alignat}
\end{subequations}

The following lemma states
additional identities used throughout the paper.
Its proof is found in 
\cite[Lemma 5.7]{christiansen2020discrete}.
\begin{lemma}\label{lem:iden}
  For a sufficiently smooth matrix-valued function $u$,
\begin{subequations} \label{id}
  \begin{alignat}{1}
      \label{curlid}
      s' \,(\bcurl u) \,\bn_\F \,  & = \curl_\F (u_{\F s})', \text{ for any } s
      \in \mathbb{R}^3,
      \\
      \label{id4}
      \big[(\bcurl u)'\big]_{\Fn}& = \curl_\F  u_{\FF}.  
      \intertext{If in addition $u$ is symmetric, then}
      \label{id1}
      (\Inc u )_{nn}
      &=   \incF \, u_\FF, 
      \\
            \label{id2}
      (\Inc u)_{\Fn} &= \curl_\F\big[(\bcurl u)'\big]_{\FF},
      \\
      \label{id3}
      \tr_\F \bcurl u
      & =-\curl_\F (u_{\Fn})'.  
     \intertext{For a sufficiently smooth vector-valued function $v$,}
      2 (\bcurl \varepsilon (v))'
      & = \Grad \, \bcurl v \label{more1},
      \\
      2\left[(\bcurl \varepsilon (v))'\right]_{\FF}
      & =  \Grad_\F (\bcurl v)_\F \label{more2},
      \\
      \label{more3}
      \bcurl v
      & = \bn_\F (\curl_\F v_\F) + \rot_\F (v \cdot \bn_\F) + \bn_\F \times \partial_n v,
      \\
      \label{more4}
      2[\varepsilon(v)]_{n \F} & = 2 [\varepsilon(v)_{\Fn}]' 
      =
      \Grad_\F (v\cdot \bn_\F)
      + \partial_n v_\F,
      \\
      \label{more5}
      \tr_\F (\rot_\F v_\F') & = \bcurl_\F v_\F.
    \end{alignat}
\end{subequations}
\end{lemma}

\subsection{Hilbert spaces}\label{subsec:Hilbert spaces} We summarize the definitions of Hilbert spaces which we use to define the discrete spaces.
For any $T \in \calT_h$, we commonly use $\mathring{(\cdot)}$ to denote the corresponding spaces with vanishing traces; see the following two examples:
\begin{align*}
\mathring{H}({\rm div}, T): = \{v \in H({\rm div}, T):  v \cdot \bn|_{\partial T} = 0\}, \quad \mathring{H}(\bcurl, T): = \{v \in H(\bcurl, T):  v \times \bn|_{\partial T} = 0\}.
\end{align*}
In addition, for any face $F \in \Delta_2(T)$ with $T \in \calT_h$, we define the following spaces by using surface operators in Section \ref{subsec:Matrix}:
\begin{alignat*}{3}
&H({\rm div}_\F, F) := \{v \in [L^2(F)]^2: \DivF v \in L^2(F)\}, 
\mathring{H}({\rm div}_\F, F): = \{v \in H({\rm div}_\F, F):  v \cdot \bs|_{\partial F} = 0\}, \\
&H({\rm curl}_\F, F) := \{v \in [L^2(F)]^2: \curl_F v \in L^2(F)\}, 
\mathring{H}({\rm curl}_\F, F): = \{v \in H({\rm curl}_\F, F):  v \cdot \bt|_{\partial F} = 0\},\\
&H({\rm grad}_\F, F) := \{v \in L^2(F): \Grad_F v \in L^2(F)\}, 
\mathring{H}({\rm grad}_\F, F): = \{v \in H({\rm grad}_\F, F):  v|_{\partial F} = 0\},
\end{alignat*}
where $\bs$ denotes the outward unit normal of $\partial F$ and $\bt$ denotes the unit tangential of $\partial F$.

\section{Discrete complexes on Clough-Tocher splits}\label{sec:2dela}
Recall a Worsey-Farin split 
of a tetrahedron induces a Clough-Tocher
split on each of its faces.
As a result, to construct
degrees of freedom and commuting projections
for discrete three-dimensional elasticity complexes on Worsey-Farin splits,
we first derive two-dimensional discrete elasticity complexes
on Clough-Tocher splits.  Throughout this section,
$F\in \Delta_2(\calT_h)$ is a face
of the (unrefined) triangulation $\calT_h$,
and $F^{ct}$ denotes its Clough-Tocher refinement
with respect to the split point $m_F$ (arising from the 
Worsey-Farin refinement of $\calT_h$).

\subsection{de Rham complexes}
As an intermediate step to derive
elasticity complexes on $F^{ct}$,
we first state several discrete de Rham complexes
with various levels of smoothness. 
First, we define the  
N\'ed\'elec spaces (without and with boundary conditions)
on the Clough--Tocher split:
\begin{alignat*}{3}
&V_{{\rm div}, r}^1(\Fct):= \{ v \in H({\rm div}_\F, F): v|_Q \in [\cP_r(\rev{\tau})]^2, \tau \in \Fct \},  \quad && \mathring{V}^1_{{\rm div},r}(\Fct) := V_{{\rm div}, r}^1(\Fct)\cap 
\mathring{H}({\rm div}_\F, F)\\
&V_{\curl, r}^1(\Fct):= \{ v \in H(\curl_\F,F): v|_\tau \in {[\cP_r(\rev{\tau})]^2},  \tau \in \Fct \},\quad && \mathring{V}^1_{{\curl},r}(\Fct) := V_{\rm curl, r}^1(\Fct)\cap 
\mathring{H}({\curl}_\F, F),\\
&{V}_r^2(\Fct):=\{ v \in L^2(F): v|_\tau \in \cP_r(\rev{\tau}), \tau \in \Fct\},\quad &&\mathring{V}_r^2(\Fct):=V_r^2(\Fct)\cap  L^2_0(F),
\end{alignat*}
and the Lagrange spaces,
\begin{alignat*}{3}
&\Lag_r^0(\Fct) := V_r^2(\Fct)\cap H({\rm grad}_\F, F),\quad && \mathring{\Lag}_r^0(\Fct) := \Lag_r^0(\Fct)\cap \mathring{H}({\rm grad}_\F, F),\\
&\Lag_r^1(\Fct) := [\Lag_r^0(\Fct)]^2,\quad && \mathring{\Lag}_r^1(\Fct) := [\mathring{\Lag}_r^0(\Fct)]^2,\\
&\Lag_r^2(\Fct) := \Lag_r^0(\Fct),\quad && \mathring{\Lag}_r^2(\Fct) := \mathring{\Lag}_r^0(\Fct)\cap L^2_0(F).
\end{alignat*}
\rev{Note that superscripts in the notation for the
spaces refer to the order of the corresponding differential forms.}

Finally, we define the (smooth) piecewise polynomial
subspaces with $C^1$ continuity.
\begin{alignat*}{1}
&S_r^0(\Fct) :=\{v\in \Lag_r^0(\Fct):\ {\rm grad}_\F\,v \in {\Lag_{r-1}^1(\Fct)} \},\\
&\mathring{S}_r^0(\Fct) :=\{v\in \mathring{\Lag}_r^0(\Fct):\ {\rm grad}_\F\,v \in {\mathring{\Lag}_{r-1}^1(\Fct)} \},\\
%
%
& \mathcal{R}_r^0(\Fct) := \{v \in S_r^0(\Fct) : v|_{\partial F} = 0\}.
%
\end{alignat*}
\rev{The first space $S^0_r(\Fct)$ is the so-called 
Hsieh-Clough-Tocher $C^1$ finite element space \cite{clough1965finite}.}
Several combinations of these spaces form exact sequences, 
as summarized in the following theorem.
\begin{thm}\label{2dseqs}
Let \rev{$r \geq 3$}. The following sequences are exact \cite{arnold1992quadratic,FuGuzman}.
\begin{subequations}
\begin{alignat}{4}
&\mathbb{R}\
{\xrightarrow{\hspace*{0.5cm}}}\
{{\Lag}}_{r}^0(\Fct)\
&&\stackrel{{\rm grad}_\F}{\xrightarrow{\hspace*{0.5cm}}}\
{{V}}_{\curl,r-1}^1(\Fct)\
&&\stackrel{\curl_\F}{\xrightarrow{\hspace*{0.5cm}}}\
{{V}}_{r-2}^2(\Fct)\
&&\xrightarrow{\hspace*{0.5cm}}\
 0,\label{alfseq1}\\
 &\mathbb{R}\
{\xrightarrow{\hspace*{0.5cm}}}\
{S}_{r}^0(\Fct)\
&&\stackrel{{\rm grad}_\F}{\xrightarrow{\hspace*{0.5cm}}}\
{\Lag}_{r-1}^1(\Fct)\
&&\stackrel{\curl_\F}{\xrightarrow{\hspace*{0.5cm}}}\
{V}_{r-2}^2(\Fct)\
&&\xrightarrow{\hspace*{0.5cm}}\
 0,\label{alfseq2}\\
%
&0\
{\xrightarrow{\hspace*{0.5cm}}}\
{\mathring{\Lag}}_{r}^0(\Fct)\
&&\stackrel{{\rm grad}_\F}{\xrightarrow{\hspace*{0.5cm}}}\
{\mathring{V}}_{\curl,r-1}^1(\Fct)\
&&\stackrel{\curl_\F}{\xrightarrow{\hspace*{0.5cm}}}\
{\mathring{V}}_{r-2}^2(\Fct)\
&&\xrightarrow{\hspace*{0.5cm}}\
 0,\label{2dbdryseq1}\\
 &0\
{\xrightarrow{\hspace*{0.5cm}}}\
\mathring{S}_{r}^0(\Fct)\
&&\stackrel{{\rm grad}_\F}{\xrightarrow{\hspace*{0.5cm}}}\
\mathring{\Lag}_{r-1}^1(\Fct)\
&&\stackrel{\curl_\F}{\xrightarrow{\hspace*{0.5cm}}}\
\mathring{V}_{r-2}^2(\Fct)\
&&\xrightarrow{\hspace*{0.5cm}}\
 0.\label{2dbdryseq2}
 %
 \end{alignat}
 \end{subequations}
 \end{thm}
Theorem \ref{2dseqs} has an alternate form that follows from a rotation of the coordinate axes, where the operators ${\rm grad}_F$ and $\curl_F$ are replaced by ${\rm rot}_F$ and ${\rm div}_F$, respectively.

\begin{cor}\label{cor:rotdiv}
Let \rev{$r \geq 3$}. The following sequences are exact \cite{arnold1992quadratic,FuGuzman}.
\begin{subequations}
\begin{alignat}{4}
&\mathbb{R}\
{\xrightarrow{\hspace*{0.5cm}}}\
{{\Lag}}_{r}^0(\Fct)\
&&\stackrel{{\rm rot}_\F}{\xrightarrow{\hspace*{0.5cm}}}\
{{V}}_{{\rm div},r-1}^1(\Fct)\
&&\stackrel{{\rm div}_\F}{\xrightarrow{\hspace*{0.5cm}}}\
{{V}}_{r-2}^2(\Fct)\
&&\xrightarrow{\hspace*{0.5cm}}\
 0,\label{altalfseq1}\\
 &\mathbb{R}\
{\xrightarrow{\hspace*{0.5cm}}}\
{S}_{r}^0(\Fct)\
&&\stackrel{{\rm rot}_\F}{\xrightarrow{\hspace*{0.5cm}}}\
{\Lag}_{r-1}^1(\Fct)\
&&\stackrel{{\rm div}_\F}{\xrightarrow{\hspace*{0.5cm}}}\
{V}_{r-2}^2(\Fct)\
&&\xrightarrow{\hspace*{0.5cm}}\
 0,\label{altalfseq2}\\
&0\
{\xrightarrow{\hspace*{0.5cm}}}\
{\mathring{\Lag}}_{r}^0(\Fct)\
&&\stackrel{{\rm rot}_\F}{\xrightarrow{\hspace*{0.5cm}}}\
{\mathring{V}}_{\Div,r-1}^1(\Fct)\
&&\stackrel{{\rm div}_\F}{\xrightarrow{\hspace*{0.5cm}}}\
{\mathring{V}}_{r-2}^2(\Fct)\
&&\xrightarrow{\hspace*{0.5cm}}\
 0,\label{alt2dbdryseq1}\\
 &0\
{\xrightarrow{\hspace*{0.5cm}}}\
\mathring{S}_{r}^0(\Fct)\
&&\stackrel{{\rm rot}_\F}{\xrightarrow{\hspace*{0.5cm}}}\
\mathring{\Lag}_{r-1}^1(\Fct)\
&&\stackrel{{\rm div}_\F}{\xrightarrow{\hspace*{0.5cm}}}\
\mathring{V}_{r-2}^2(\Fct)\
&&\xrightarrow{\hspace*{0.5cm}}\
  0.\label{alt2dbdryseq2}
 \end{alignat}
 \end{subequations}
\end{cor}

\subsection{Elasticity complexes}
In order to construct elasticity sequences in three dimensions, 
we need some elasticity complexes on the two-dimensional Clough-Tocher splits. 
{The main results of this section are very similar to the ones found \cite{christiansen2022finite} (with spaces slightly different)  and can be proved with the techniques there}. However, 
to be self-contained, we provide the \rev{proof of the main result, Theorem~\ref{thm:2delaseq} in an appendix}. 
\rev{Let $\bV_2$ denote the plane $n^\perp$ where $n$ is a unit normal to $\Fct$; clearly $\bV_2$ is isomorphic to $\bbR^2$. Then the} two-dimensional elasticity complexes
utilize these:
\begin{subequations}
\begin{alignat}{1}
    & \mathring{Q}_{{\rm inc},r}^1(\Fct) :=\{v\in \mathring{\Lag}_r^1(\Fct) \otimes \bV_2:\   \curl_\F v \in {\mathring{V}_{{\rm curl},r-1}^1(\Fct)} \}, \\
    \label{eqn:Qr0b}
   & \mathring{Q}_{{\rm inc},r}^{1,s}(\Fct) := \{\sym(u): u \in \mathring{Q}_{{\rm inc}, r}^1(\Fct) \}, \\
   \label{eqn:Qr1}
& Q^1_{r} (\Fct) := \{u \in V_{{\rm div}, r}^1(\Fct) \otimes \bV_2: {\rm skew}(u) = 0\}, \\
\label{eqn:Qr1tilde}
 & \Tilde{Q}^1_{r} (\Fct):= \{u \in \Lag_{r}^1(\Fct) \otimes \bV_2: {\rm skew}(u) = 0\} \subset Q^1_{r}(\Fct), \\
  \label{eqn:Qr2b}
& \mathring{Q}^2_{r} (\Fct) := \{u \in V^2_{r}(\Fct): u \perp \cP_{1}(F)\}.
\end{alignat}
\end{subequations}
We further let $Q_r^{\perp}$ be the subspace
of $Q_r^1(\Fct)$ that is $L^2(F)$-orthogonal to $\Tilde{Q}^1_r(\Fct)$.
We then have $Q_r^1(\Fct) = Q_r^\perp \oplus \tilde{Q}^1_r(\Fct)$, 
and
\begin{equation} \label{eqn:}
    \dim Q_r^{\perp} = \dim Q^1_{r}(\Fct) - \dim \Tilde{Q}^1_{r}(\Fct).
\end{equation}

\begin{lemma}[Lemma 5.8 in \cite{christiansen2020discrete}]\label{lem:incFiden}
Let $u$ be a sufficiently smooth matrix-valued function,
and let $\phi$ be a smooth scalar-valued function.
Then there holds the following integration-by-parts identity:
\begin{equation}\label{eqn:incF}
    \int_{F} (\incF \, u) \, \phi = \int_{F} u \colon \airy(\phi) + \int_{\partial F} (\curl_\F u)\bt \, \phi ds + \int_{\partial F} u \bt \cdot (\rot_\F \phi)'.
\end{equation}
Consequently, if $u \in \mathring{Q}_{{\rm inc},r-1}^{1}(\Fct)$ is symmetric and $\phi \in \cP_1(F)$, 
then $\int_{F} (\incF \, u) \, \phi =0$.
\end{lemma}
%

The next theorem is the main result
of \rev{this} section, where exact local discrete elasticity complexes
are presented on Clough-Tocher splits.
Its proof is given in \rev{Appendix \ref{apdix:sec3}}.

\begin{thm} \label{thm:2delaseq}
    Let $r \ge 3$. The following elasticity sequences are exact.
\begin{alignat}{4}
&0 \
{\xrightarrow{\hspace*{0.5cm}}}\
\mathring{{S}}_{r+1}^0(\Fct) \otimes \bV_2 \
&&\stackrel{{\varepsilon}_\F}{\xrightarrow{\hspace*{0.5cm}}}\
\mathring{Q}_{{\rm inc},r}^{1,s}(\Fct)\
&&\stackrel{\incF}{\xrightarrow{\hspace*{0.5cm}}}\
\mathring{Q}^2_{r-2} (\Fct) \
&&\xrightarrow{\hspace*{0.5cm}}\
 0,\label{elaseqsvenb}
\end{alignat} 
\begin{alignat}{4}
&\cP_1(F)\
\stackrel{\subset}{\xrightarrow{\hspace*{0.5cm}}}\
{{S}}_{r}^0(\Fct)\
&&\stackrel{\airy}{\xrightarrow{\hspace*{0.5cm}}}\
Q^1_{r-2}(\Fct)\
&&\stackrel{\DivF}{\xrightarrow{\hspace*{0.5cm}}}\
V^2_{r-3}(\Fct) \otimes \bV_2 \
&&\xrightarrow{\hspace*{0.5cm}}\
 0.\label{elaseqairy}
\end{alignat}
\end{thm}

\subsection{Dimension counts}
We summarize the dimension counts of the 
discrete spaces on the Clough-Tocher split in Table 
\ref{tab:2DDim} which will be used in the construction 
elasticity complex in three dimensions.
These dimensions are mostly found in \cite{guzman2022exact}
and follow from Theorem~\ref{2dseqs}
and the rank-nullity theorem. Likewise, the dimension
of $Q^1_r(F^{ct})$ follows from Theorem~\ref{thm:2delaseq}.

\begin{table}[ht]
\caption{\label{tab:2DDim}Dimension counts of the canonical (two--dimensional) N\'ed\'elec, Lagrange, and smooth spaces 
with respect to the Clough--Tocher split.  Here, $\dim V_{{\rm div},r}^1(\Fct) = \dim V_{\curl,r}^1(\Fct) =:\dim V_r^1(\Fct)$}
{\scriptsize 
\begin{tabular}{c|ccc}
& $k=0$ & $k=1$ & $k=2$\\
\hline
$\dim V_r^k(\Fct)$ & \rev{---}  & $3(r+1)^2$ & $\frac{3}{2}(r+1)(r+2)$\\
$\dim \mathring{V}_r^k(\Fct)$ &  \rev{---}  & $3r(r+1)$ & $\frac{3}{2}(r+1)(r+2) - 1$\\
$\dim \Lag_r^k(\Fct)$ & $ \frac{1}{2}(3r^2 + 3r + 2)$ & $3r^2 + 3r + 2$ & $\frac{1}{2}(3r^2 + 3r + 2)$\\
$\dim \mathring{\Lag}_r^k(\Fct)$ & $\frac{1}{2}(3r^2 - 3r + 2)$ & $3r^2 - 3r + 2$ & $\frac{3}{2}r(r-1)$\\
$\dim S_r^k(\Fct)$ & $ \frac{3}{2}(r^2 - r +2)$ & \rev{---}  & \rev{---} \\
%
%
$\dim \mathcal{R}_{r}^k(\Fct)$ & $\frac{3}{2}(r-1)(r-2)$ \cite{ALphd} &
---
&--- \\
$\dim Q_{r}^k(\Fct)$ & --- &
$\frac{3}{2}(3r^2 +5r+2)$
&---
\end{tabular}
}
\end{table}

\section{Local discrete sequences on Worsey-Farin splits} \label{sec:localcomplex}
\subsection{de Rham complexes}\label{subsec:localDeRhamy}
Similar to the two-dimensional setting in Section \ref{sec:2dela},
the starting point to construct discrete 3D elasticity complexes
are the de Rham complexes consisting of piecewise polynomial spaces.
The N\'ed\'elec spaces with respect to the local Worsey-Farin split $\WFT$
are given as
\begin{alignat*}{2}
&V_r^1(\WFT):=[\cP_r(\WFT)]^3 \cap H(\curl,T),\qquad
&& \mathring{V}_r^1(\WFT):= V_r^1(\WFT) \cap \mathring{H}(\curl,T),\\
&V_r^2(\WFT):=[\cP_r(\WFT)]^3 \cap H({\rm div},T),\qquad
&&\mathring{V}_r^2(\WFT):= V_r^2(\WFT) \cap \mathring{H}({\rm div},T),\\
 &V_r^3(\WFT):=\cP_r(\WFT), &&\mathring{V}_r^3(\WFT):=V_r^3(\WFT)\cap L^2_0(T).
\end{alignat*}
The Lagrange spaces on $\WFT$ are defined by
\begin{alignat*}{2}
    & \Lag_{r}^0(\WFT) := \cP_r(\WFT) \cap H^1(T),\qquad
    && \mathring{\Lag}_{r}^0(\WFT) := \Lag_r^0(\WFT) \cap \mathring{H}^1(T), \\
    & \Lag_{r}^1(\WFT) := [\Lag_{r}^0(\WFT)]^3, && \mathring{\Lag}_{r}^1(\WFT) := [\mathring{\Lag}_{r}^0(\WFT)]^3, \\
    & \rev{\Lag_{r}^2(\WFT):= \Lag_{r}^1(\WFT),} && \rev{\mathring{\Lag}_{r}^2(\WFT) :=\mathring{\Lag}_{r}^1(\WFT),}
\end{alignat*}
and the discrete spaces with additional smoothness
are 
\begin{alignat*}{2}
     & S_{r}^0(\WFT):=\{u \in \Lag_{r}^0(\WFT): \Grad \, u \in \Lag_{r-1}^1(\WFT)\},\\
     & \mathring{S}_{r}^0(\WFT):=\{u \in \mathring{\Lag}_{r}^0(\WFT): \Grad \, u \in \mathring{\Lag}_{r-1}^1(\WFT)\}, &&\\
     & S_{r}^1(\WFT):=\{u \in \Lag_{r}^1(\WFT): \bcurl u \in \Lag_{r-1}^1(\WFT)\},\\
     & \mathring{S}_{r}^1(\WFT):=\{u \in \mathring{\Lag}_{r}^1(\WFT): \bcurl u \in \mathring{\Lag}_{r-1}^1(\WFT)\}.
\end{alignat*}
We also define the intermediate spaces
\begin{align*}
    & {\cV}_r^2(\WFT):=\{v \in V_r^2(\WFT): \rev{v \times \bn|_F} \text{ is continuous on each } F \in \Delta_2(T)\}, \\
    & \mathring{\cV}_r^2(\WFT) := \{v \in \cV_r^2(\WFT): \rev{v \cdot \bn|_F}=0 \text{ on each } F \in \Delta_2(T)\}, \\
    & {\cV}_r^3(\WFT):=\{q \in V_r^3(\WFT): \rev{q|_F} \text{ is continuous on each } F \in \Delta_2(T)\}, \\
    & \mathring{\cV}_r^3 := \cV_r^3(\WFT) \cap L_0^2(T).
\end{align*}
and note that
\begin{alignat*}{2}
&S_r^0(\WFT) \subset \Lag^0_r(\WFT),\qquad
&& \rev{S_r^1(T^{wf})}\subset \Lag^1_r(\WFT)\subset V_r^1(\WFT),\\
&\Lag^2_r(\WFT)\subset \cV_r^2(\WFT)\subset V_r^2(\WFT),\quad
&&\cV_r^3(\WFT)\subset V_r^3(\WFT),
\end{alignat*}
with similar inclusions holding for the analogous spaces
with boundary conditions.

The next lemma summarizes
the exactness properties
of several (local) complexes using these spaces.
Its proof is found in 
 \cite[Theorem 3.1-3.2]{guzman2022exact}.
 \begin{lemma}\label{lem:localseq}
The following sequences are exact for any $r \ge 3$.
\begin{subequations}
\begin{equation} \label{eqn:seq0}
     \mathbb{R}  \xrightarrow{\subset} \Lag_{r}^0(\WFT) \xrightarrow{\Grad} V_{r-1}^1(\WFT) \xrightarrow{\bcurl} V_{r-2}^2(\WFT) \xrightarrow{\Div} V_{r-3}^3(\WFT) \rightarrow 0,
 \end{equation}
 \begin{equation} \label{eqn:seq0b}
     0  \rightarrow \mathring{\Lag}_{r}^0(\WFT) \xrightarrow{\Grad} \mathring{V}_{r-1}^1(\WFT) \xrightarrow{\bcurl} \mathring{V}_{r-2}^2(\WFT) \xrightarrow{\Div} \mathring{V}_{r-3}^3(\WFT) \rightarrow 0,
 \end{equation}
\begin{equation} \label{eqn:seq1}
     \mathbb{R}  \xrightarrow{\subset} S_{r}^0(\WFT) \xrightarrow{\Grad} \Lag_{r-1}^1(\WFT) \xrightarrow{\bcurl} V_{r-2}^2(\WFT) \xrightarrow{\Div} V_{r-3}^3(\WFT) \rightarrow 0,
 \end{equation}
 \begin{equation} \label{eqn:seq1b}
     0  \rightarrow \mathring{S}_{r}^0(\WFT) \xrightarrow{\Grad} \mathring{\Lag}_{r-1}^1(\WFT) \xrightarrow{\bcurl} \mathring{\cV}_{r-2}^2(\WFT) \xrightarrow{\Div} \mathring{V}_{r-3}^3(\WFT) \rightarrow 0.
 \end{equation}
 \begin{equation} \label{eqn:seq2}
     \mathbb{R}  \xrightarrow{\subset} S_{r}^0(\WFT) \xrightarrow{\Grad} S_{r-1}^1(\WFT) \xrightarrow{\bcurl} \Lag_{r-2}^2(\WFT) \xrightarrow{\Div} V_{r-3}^3(\WFT) \rightarrow 0.
 \end{equation}
 \begin{equation} \label{eqn:seq2b}
    0 \rightarrow \mathring{S}_{r}^0(\WFT) \xrightarrow{\Grad} \mathring{S}_{r-1}^1(\WFT) \xrightarrow{\bcurl} \mathring{\Lag}_{r-2}^2(\WFT) \xrightarrow{\Div} \mathring{\cV}_{r-3}^3(\WFT) \rightarrow 0.
 \end{equation}
 \end{subequations}
\end{lemma}

\subsection{Dimension counts}
The dimensions of the spaces in Section \ref{subsec:localDeRhamy}
are summarized in Table \ref{tab:VLDim}.
These counts essentially from Lemma \ref{lem:localseq} and the rank-nullity theorem;
see \cite{guzman2022exact} for details.
\begin{table}[ht]
\caption{\label{tab:VLDim}Dimension counts of the canonical N\'ed\'elec, Lagrange spaces and smoother spaces on a WF split. Here \rev{$a^+ = \max(a, 0)$.}}
{\scriptsize 
\begin{tabular}{c|cccc}
& $k=0$ & $k=1$ & $k=2$ & $k=3$\\
\hline
$V_r^k(\WFT)$ & $(2r+1)(r^2+r+1)$ & $2 (r + 1) (3 r^2 + 6 r + 4)$ & $3 (r + 1) (r + 2) (2 r + 3)$ & $2 ( r+1) (r+2) (r+3)$\\
$\mathring{V}_r^k(\WFT)$ & $(2r-1)(r^2-r+1)$ & $ 2 (r + 1) (3 r^2 + 1)$ & $3 (r+1) (r+2) (2r+1)$ & $2 r^3 + 12 r^2 + 22 r + 11$\\
$\Lag_r^k(\WFT)$ & $(2r+1)(r^2+r+1)$ & $3(2r+1)(r^2+r+1)$ & $3(2r+1)(r^2+r+1)$ & $(2r+1)(r^2+r+1)$\\
$ \mathring{\Lag}_r^k(\WFT)$ & $(2r-1)(r^2-r+1)$ & $3 (2 r - 1) (r^2 - r + 1)$ & $ 3 (2 r - 1) (r^2 - r + 1)$ & $(r - 1) (2 r^2 - r + 2)$ \\
$\mathring{\cV}_r^k(\WFT)$ & --- & --- & $6r^3 + 21r^2 + 9r + 2$ & {$2r^3 +12r^2+10r+3$}\\
$S_r^k(\WFT)$ & $ 2r^3-6r^2+10r-2$ & $3r(2r^2-3r+5)$ & $6r^3+8r+2$ & $(2r+1)(r^2+r+1)$\\
$\mathring{S}_r^0(\WFT)$ & $\big(2(r-2)(r-3)(r-4)\big)^+$ & $\big(3(2r-3)(r-2)(r-3)\big)^+$ & $\big(2(r-2)(3r^2-6r+4)\big)^+$ & $(r-1)(2r^2-r+2)$
\end{tabular}
}
\end{table}



\subsection{Elasticity complex for stresses with weakly imposed symmetry}
In this section we will apply Proposition \ref{prop:exactpattern} to
the de-Rham sequences on Worsey-Farin splits. This gives rise to a
derived complex useful for analyzing mixed methods for elasticity with
weakly imposed stress symmetry. From this intermediate step, an
elasticity sequence with strong symmetry will readily follow.
We start with the following definition and lemma.

\begin{definition} \label{def:mu}
    Let $\mu \in \mathring{\Lag}_1^0(\WFT)$
    be the unique continuous, piecewise linear polynomial that vanishes on $\partial T$ and takes the value $1$ at the incenter of $T$.
\end{definition}

\begin{lemma} \label{lem:mapproperty}\  
    \begin{enumerate}
    \item  The map $\Xi : \Lag_{r}^1(\WFT) \otimes \bV \rightarrow \Lag_{r}^2(\WFT)\otimes \bV $ is a bijection. 
    \item The following inclusions hold $\vskw \, (V_{r-2}^2(\WFT) \otimes \bV) \subset V_{r-2}^3(\WFT) \otimes \bV$ and \\
    $\vskw \, (\mathring{\cV}_{r-2}^2(\WFT) \otimes \bV) \subset \cV_{r-2}^3(\WFT) \otimes \bV$, \rev{for any $r \ge 3$}. 
    \item The mappings $\vskw : V^2_{r-2}(\WFT) \otimes \bV \rightarrow V^3_{r-2}(\WFT) \otimes \bV$ and $\vskw : \mathring{\cV}^2_{r-2}(\WFT) \otimes \bV \rightarrow \cV^3_{r-2}(\WFT) \otimes \bV$ are both surjective, \rev{for any $r \ge 3$}.
      \end{enumerate}
\end{lemma}

\begin{proof}

    
    Both (1) and (2) are trivial to verify and hence we only prove (3). 
    \rev{For any $r \ge 3$, let} $v \in V^3_{r-2}(\WFT) \otimes \bV$. By the exactness of \eqref{eqn:seq2}, there exists a function $z \in \Lag_{r-2}^2(\WFT)\otimes \bV$ such that $\Div z = v$. Since $\Xi$ is a bijection from $ \Lag_{r-2}^1(\WFT) \otimes \bV$ to $\Lag_{r-2}^2(\WFT)\otimes \bV$, we have $q= \Xi^{-1} z \in \Lag_{r-2}^1(\WFT) \otimes \bV$.  Thus, by setting 
    $w = \bcurl q \in V_{r-2}^2(\WFT) \otimes \bV$ we obtain  
    \[
    2 \vskw (w) = 2 \vskw \, \bcurl (q)= 2 \vskw \, \bcurl (\Xi^{-1} z ) = \Div \Xi (\Xi^{-1} z) = v,
    \]
    where we used \eqref{eqn:iden1}.  We conclude
    $\vskw : V^2_{r-2}(\WFT) \otimes \bV \rightarrow V^3_{r-2}(\WFT) \otimes \bV$ is a surjection.

    We now prove the analogous result with boundary condition. Let $v \in \cV^3_{r-2}(\WFT) \otimes \bV$, and let $M \in \mathbb{M}_{3\times 3}$ be a constant matrix
    such that $\int_{T} 2 \vskw \, M = \frac{1}{ \int_T \mu}{\int_T v}$. 
    Then, by taking $\tilde{w}= \mu M$, we have  $\Tilde{w} \in \mathring{\cV}_{1}^2(\WFT) \otimes \bV$ with $\int_{T} 2 \vskw \, \Tilde{w} = \int_T v$. Therefore, we have $v-2\,\vskw(\Tilde{w}) \in \mathring{\cV}^3_{r-2}(\WFT) \otimes \bV$ and the exactness of \eqref{eqn:seq2b} 
    yields the existence of $z \in \mathring{\Lag}_{r-1}^2(\WFT) \otimes \bV$, such that $\Div z = v- 2\vskw(\Tilde{w})$. Let $q = \Xi^{-1} z \in \mathring{\Lag}_{r-1}^1(\WFT) \otimes \bV$, and from \eqref{eqn:seq1b}, we have $w:=\bcurl(q) + \Tilde{w} \in \mathring{\cV}_{r-2}^2(\WFT) \otimes \bV$. Finally, using \eqref{eqn:iden1} 
    \begin{align*}
        2\vskw(w)= 2\vskw \, \bcurl(\Xi^{-1}z)+2\vskw(\Tilde{w}) =  \Div z +2\vskw(\Tilde{w}) =v.
    \end{align*}
    This shows the surjectivity of $\vskw : \mathring{\cV}^2_{r-2}(\WFT) \otimes \bV \rightarrow \cV^3_{r-2}(\WFT) \otimes \bV$, thus completing the proof.
\end{proof}

Using the complexes \eqref{eqn:seq1}-\eqref{eqn:seq2b} and the two identities \eqref{eqn:iden1}-\eqref{eqn:iden2}, we construct the following commuting diagrams:
{ 
\begin{equation}\label{eqn:diagram1}
    \begin{tikzcd}
 S_{r+1}^0(\WFT) \!\otimes\! \bV \arrow{r}{\Grad} & S_{r}^1(\WFT) \!\otimes\! \bV \arrow{r}{\bcurl} & \Lag_{r-1}^2(\WFT) \!\otimes\! \bV \arrow{r}{\Div} & V_{r-2}^3(\WFT) \!\otimes\! \bV  \arrow{r} & [-1em] 0
 \\ 
S_{r}^0(\WFT) \!\otimes\! \bV \arrow{r}{\Grad}  \arrow[swap]{ur}{-\mskw} & \Lag_{r-1}^1(\WFT) \!\otimes\! \bV \arrow{r}{\bcurl} \arrow[swap]{ur}{\Xi} & V_{r-2}^2(\WFT) \!\otimes\! \bV \arrow{r}{\Div} \arrow[swap]{ur}{2\vskw}& V_{r-3}^3(\WFT) \!\otimes\! \bV \arrow{r} &[-1em] 0,
\end{tikzcd}
\end{equation}
}
\begin{equation}\label{eqn:diagram1b}
    \begin{tikzcd}
 \mathring{S}_{r+1}^0(\WFT) \!\otimes\! \bV \arrow{r}{\Grad} & \mathring{S}_{r}^1(\WFT) \!\otimes\! \bV \arrow{r}{\bcurl} & \mathring{\Lag}_{r-1}^2(\WFT) \!\otimes\! \bV \arrow{r}{\Div} & \cV_{r-2}^3(\WFT) \!\otimes\! \bV \arrow{r}{\int} & [-1em]\rev{\bbR} \\ 
\mathring{S}_{r}^0(\WFT) \!\otimes\! \bV \arrow{r}{\Grad}  \arrow[swap]{ur}{-\mskw} & \mathring{\Lag}_{r-1}^1(\WFT) \!\otimes\! \bV \arrow{r}{\bcurl} \arrow[swap]{ur}{\Xi} & \mathring{\cV}_{r-2}^2(\WFT) \!\otimes\! \bV \arrow{r}{\Div} \arrow[swap]{ur}{2\vskw}& \mathring{V}_{r-3}^3(\WFT) \!\otimes\! \bV \arrow{r} & [-1em]0.
\end{tikzcd}
\end{equation}
Note that the top sequence of \eqref{eqn:diagram1b} is slightly different from \eqref{eqn:seq2b}, as the mean-value constraint
is not imposed on $\cV_{r-2}(\WFT)\otimes \mathbb{V}$.
This is due to the surjective property of the mapping
$\vskw: (\mathring{\cV}_{r-2}^2(\WFT) \otimes \bV) \to \cV_{r-2}^3(\WFT) \otimes \bV$
established in Lemma \ref{lem:mapproperty}.

\begin{thm} \label{thm:preseq}
    The following sequences are exact for any $r \ge 3$:
    \begin{equation} \label{eqn:preseq}
            \begin{bmatrix}
            S_{r+1}^0(\WFT) \!\otimes\! \bV \\
            S_{r}^0(\WFT \!\otimes\! \bV)
            \end{bmatrix}
            \xrightarrow{[\Grad, -\mskw]} S_{r}^1(\WFT) \!\otimes\! \bV \xrightarrow{\inc} {V}_{r-2}^2(\WFT) \!\otimes\! \bV \xrightarrow{
              \left[\begin{smallmatrix}
                  2\vskw \\ \Div
                \end{smallmatrix}\right]} 
            \begin{bmatrix}
              V^3_{r-2}(\WFT) \!\otimes\! \bV \\ {V}_{r-3}^3(\WFT) \!\otimes\! \bV
            \end{bmatrix}.
    \end{equation}
    
    \begin{equation} \label{eqn:preseqb}
            \begin{bmatrix}
            \mathring{S}_{r+1}^0(\WFT) \otimes \bV \\
            \mathring{S}_{r}^0(\WFT \otimes \bV)
            \end{bmatrix}
            \xrightarrow{[\Grad, -\mskw]} \mathring{S}_{r}^1(\WFT) \otimes \bV \xrightarrow{\inc} \mathring{\cV}_{r-2}^2(\WFT) \otimes \bV \xrightarrow{ \left[\begin{smallmatrix}
            2\vskw \\ \Div \end{smallmatrix}\right]} 
            \begin{bmatrix}
              \cV^3_{r-2}(\WFT) \otimes \bV \\ \mathring{V}_{r-3}^3(\WFT) \otimes \bV
            \end{bmatrix}.
    \end{equation}
Moreover, the last operator in \eqref{eqn:preseq} is surjective.
\end{thm}

\begin{proof}
Lemma \ref{lem:mapproperty} tells us that $\Xi : \Lag_{r-1}^1(\WFT) \otimes \bV \rightarrow \Lag_{r-1}^2(\WFT)\otimes \bV $ is a bijection. With the exactness of \eqref{eqn:seq1}-\eqref{eqn:seq2b} for $r\ge3$ and Proposition \ref{prop:exactpattern}, we see that these two sequences are exact. The surjectivity of the last map is guaranteed by Proposition \ref{prop:exactpattern} and Lemma \ref{lem:mapproperty}.
\end{proof}

\subsection{Elasticity sequence} \label{subsec:ela~seq}
Now we are ready to describe the local discrete elasticity 
sequence on Worsey-Farin splits. 
The discrete elasticity complexes with strong symmetry are formed by the following spaces:
    \begin{align*}
        & U_{r+1}^0(\WFT) = S_{r+1}^0(\WFT) \otimes \bV, && \mathring{U}_{r+1}^0(\WFT) = \mathring{S}_{r+1}^0(\WFT) \otimes \bV, \\
        & U_{r}^1(\WFT) = \{\sym(u): u \in S_{r}^1(\WFT) \otimes \bV\}, && \mathring{U}_{r}^1(\WFT) = \{\sym (u): u \in \mathring{S}_{r}^1(\WFT) \otimes \bV\}, \\
        & U_{r-2}^2(\WFT) = \{u \in V_{r-2}^2(\WFT) \otimes \bV: \skw \, u =0\}, && \mathring{U}_{r-2}^2(\WFT) = \{u \in \mathring{\cV}_{r-2}^2(\WFT) \otimes \bV: \skw \,  u =0\}, \\
        & U_{r-3}^3(\WFT) = V_{r-3}^3(\WFT) \otimes \bV, && \mathring{U}_{r-3}^3(\WFT) = \{u \in V_{r-3}^3(\WFT) \otimes \bV: u \perp \Rig \},
    \end{align*}
where we recall $\Rig$, defined in \eqref{eqn:rigid1}, is the space of rigid body displacements.

\begin{thm} \label{thm:elseq}
    The following two sequences are discrete elasticity complexes and are exact for $r \ge 3$: 
     \begin{equation}\label{eqn:elseq}
         \Rig  \rightarrow {U}_{r+1}^0(\WFT) \xrightarrow{\varepsilon} {U}_{r}^1(\WFT) \xrightarrow{\Inc} {U}_{r-2}^2(\WFT) \xrightarrow{\Div} {U}_{r-3}^3(\WFT) \rightarrow 0,
    \end{equation}
    and 
    \begin{equation}\label{eqn:elseqb}
         0  \rightarrow \mathring{U}_{r+1}^0(\WFT) \xrightarrow{\varepsilon} \mathring{U}_{r}^1(\WFT) \xrightarrow{\Inc} \mathring{U}_{r-2}^2(\WFT) \xrightarrow{\Div} \mathring{U}_{r-3}^3(\WFT) \rightarrow 0.
    \end{equation}
\end{thm}

\begin{proof}
    We first show that \eqref{eqn:elseq} is a complex. In order to do this, it suffices 
    to show the operators map the space they are acting on into the subsequent space. 
    To this end, let $u \in U_{r+1}^0(\WFT)$, then by \eqref{eqn:seq2} we have $\Grad \,(u) \in S_r^1(\WFT) \otimes \bV$. Hence, $\varepsilon(u) = \sym \, \Grad \,(u) \in U_r^1(\WFT)$. Now let $u \in U_r^1(\WFT)$ which implies that $u = \sym(w)$ with $w \in S_r^1(\WFT) \otimes \bV$. Thus by \eqref{eqn:iden3} we have $\inc u = \inc w \in V_{r-2}^2(\WFT) \otimes \bV$ and $\skw(u) = 0$ due to \eqref{eqn:iden4}. Therefore, there holds $\inc (u) \in {U}_{r-2}^2(\WFT) $. Finally, for any $u \in {U}_{r-2}^2(\WFT) \subset V_{r-2}^2(\WFT) \otimes \bV$, $\Div u \in V_{r-3}^3(\WFT) \otimes \bV$. 

    Next, we prove exactness of the complex \eqref{eqn:elseq}. Let $w \in {U}_{r-3}^3(\WFT)$ and consider $(0,w) \in [V_{r-2}^3(\WFT) \otimes \bV] \times [V_{r-3}^3(\WFT) \otimes \bV]$.
    Due to the exactness of \eqref{eqn:preseq} in Theorem \ref{thm:preseq}, there exists $v \in V_{r-2}^2(\WFT) \otimes \bV$ such that $\Div v = w$ 
    and $2 \vskw(v)=0$. Thus, $v \in {U}_{r-2}^2(\WFT)$.  
    
    Now let $w \in {U}_{r-2}^2(\WFT)$ with $\Div w=0$.  Then by the exactness of \eqref{eqn:preseq}, we have the existence of $v \in S_{r}^1(\WFT) \otimes \bV$ such that $\inc v= w$.
    Setting $u=\sym(v) \in {U}_{r}^1(\WFT)$ yields $\Inc u=w$ by   \eqref{eqn:iden3}. 
    
    Finally, let $w \in {U}_{r}^1(\WFT)$ with $\Inc w=0$. Then $w=\sym(v)$ for some $v \in S_{r}^1(\WFT) \otimes \bV$ and with \eqref{eqn:iden3}, $\inc v=\inc w = 0$. Due to the exactness of \eqref{eqn:preseq}, we could find $(u,z) \in [S_{r+1}^0(\WFT) \otimes \bV] \times [S_{r}^0(\WFT) \otimes \bV]$ such that $v=\Grad \, u-\mskw(z)$. Therefore, $\varepsilon(u)=\sym(v)= w$.

    We can prove that \eqref{eqn:elseqb} is a complex and it is exact very similar to above. The main difference is the surjectivity of the last map which we prove now.  Let $w \in \mathring{U}_{r-3}^3(\WFT) \subset \mathring{V}_{r-3}^3 \otimes \bV$. Then by the exactness of \eqref{eqn:seq1b}, there exists $v \in \mathring{\cV}_{r-2}^2(\WFT) \otimes \bV$ such that $\Div v = w$. For any $c \in \mathbb{R}^3$ we have $\Grad \,(c \times x)= \mskw \, c$ and hence, using integration by parts
    \[
    \int_T 2 \vskw \, v \cdot c = \int_T v \colon \mskw \, c = \int_T v \colon \Grad \,( c \times x)=-\int_T \Div v \cdot (c \times x)= - \int_T  w \cdot (c \times x)=0, 
    \]
    where the last equality uses the fact  $w \perp \Rig$.
    Therefore, $\vskw \, v \in \mathring{\cV}_{r-2}^3(\WFT) \otimes \bV$ and  by the exactness of \eqref{eqn:seq2b}, we have \rev{an} $m \in \mathring{\Lag}_{r-1}^2(\WFT) \otimes \bV$ such that $\Div m=2 \vskw \, v$. Let $u=v-\bcurl (\Xi^{-1}m) \in \mathring{\cV}_{r-2}^2(\WFT) \otimes \bV$ and we see that $2\vskw \, u =2\vskw \, v- 2\vskw \,\bcurl (\Xi^{-1}m)=0 $ by \eqref{eqn:iden1}.  Hence, $u \in \mathring{U}_{r-2}^2(\WFT)$ and $\Div u=w$.
\end{proof}

 {When $r \ge 4$, there holds $\Rig \subset {U}_{r-3}^3(\WFT)$, so it is clear that 
\begin{equation}\label{Ur}
{U}_{r-3}^3(\WFT) = \Rig \oplus \mathring{U}_{r-3}^3(\WFT) \qquad \text{ for }  r \ge 4. 
\end{equation}
On the other hand, when $r = 3$, we need the following lemma for 
the calculation of dimensions of $\mathring{U}_{r-3}^3(\WFT)$.  
Let $P_\U$ be the $L^2$-orthogonal projection onto $U_0^3(\WFT)$ 
and let $P_\U\Rig := \{P_\U u:\ u \in  \Rig\}$. The proof of the 
following lemma is provided in the appendix. 
\begin{lemma}\label{lem:projRig}
    It holds,  
    \begin{equation}\label{U03}
    {U}_0^3(\WFT) = P_\U\Rig \oplus \mathring{U}_{0}^3(\WFT), 
    \end{equation}
   and $\dim P_\U\Rig = \dim \Rig=6$.
\end{lemma}
}

Using the exactness of the complexes \eqref{eqn:elseq}--\eqref{eqn:elseqb}
along with Table \ref{tab:VLDim}, we  calculate the dimensions of the spaces
in the next lemma.
\begin{lemma}
    When $r \ge 3$, we have:
    \begin{align}
       & \dim {U}_{r+1}^0(\WFT) = 6r^3+12r+12, && \dim \mathring{U}_{r+1}^0(\WFT) = 6r^3-36r^2+66r-36, \\
       & \dim {U}_{r}^1(\WFT) = 12r^3-9r^2+15r+6, && \dim \mathring{U}_{r}^1(\WFT)= 12r^3-63r^2+87r-18, \\
       & \dim {U}_{r-2}^2(\WFT) = 12r^3-27r^2+15r, && \dim \mathring{U}_{r-2}^2(\WFT)=12r^3-45r^2+33r+12, \\
       & \dim {U}_{r-3}^3(\WFT)=6r^3-18r^2+12r, && \dim \mathring{U}_{r-3}^3(\WFT)= 6r^3-18r^2+12r-6.
    \end{align}
\end{lemma}

\begin{proof}
    By Lemma \ref{lem:mapproperty} and the rank-nullity theorem, we have
    \begin{align*}
        \dim {U}_{r-2}^2(\WFT) & = \dim \ker (V_{r-2}^2(\WFT) \otimes \bV, \vskw) = \dim V_{r-2}^2(\WFT) \otimes \bV - \dim V_{r-2}^3(\WFT) \otimes \bV \\
        & = (6r^3-9r^2+3r) \times 3-2r(r+1)(r-1) \times 3= 12r^3-27r^2+15r,\\
        \dim \mathring{U}_{r-2}^2(\WFT) & = \dim \ker (\mathring{\cV}_{r-2}^2(\WFT) \otimes \bV, \vskw) = \dim \mathring{\cV}_{r-2}^2(\WFT) \otimes \bV - \dim \cV_{r-2}^3(\WFT) \otimes \bV \\
        & = (6(r-2)^3+21(r-2)^2+9(r-2)+2) \times 3 \\
            &~~~~ -2((r-2)^3+6(r-2)^2+5(r-2)+2) \times 3 \\
        & =18 r^3 - 45 r^2 - 9 r + 60 -(6 r^3 - 42 r + 48) = 12 r^3 - 45 r^2 + 33 r + 12.
    \end{align*}
The dimensions of ${U}_{r+1}^0(\WFT)$, $\mathring{U}_{r+1}^0(\WFT)$ and ${U}_{r-3}^3(\WFT)$ are computed similarly using the dimensions of $S_{r+1}^0(\WFT)$, 
$\mathring{S}_{r+1}^0(\WFT)$ and $V_{r-3}^3(\WFT)$. 
Also,  {using Lemma \ref{lem:projRig} when $r=3$ or \eqref{Ur} when $r \ge 4$, we obtain}
\[\dim \mathring{U}_{r-3}^3(\WFT)=  \dim {U}_{r-3}^3(\WFT)-6.\]
Using the exactness of the sequences \eqref{eqn:elseq} and \eqref{eqn:elseqb} in Theorem \ref{thm:elseq}, with the rank-nullity theorem, we have
\begin{align*}
    \dim {U}_{r}^1(\WFT)& = \dim {U}_{r+1}^0(\WFT)+\dim {U}_{r-2}^2(\WFT)- \dim {U}_{r-3}^3(\WFT)-\dim \Rig \\
    & = 12 r^3 - 9 r^2 + 15 r + 6,\\
    \dim \mathring{U}_{r}^1(\WFT)& = \dim \mathring{U}_{r+1}^0(\WFT)+\dim \mathring{U}_{r-2}^2(\WFT)-\dim \mathring{U}_{r-3}^3(\WFT) \\
    & = 12 r^3 - 63 r^2 + 87 r - 18.
\end{align*}
\end{proof}

\subsection{An equivalent characterization  of \texorpdfstring{$U^1_r(\WFT)$}{U1r} and \texorpdfstring{$\mathring{U}_r^1(\WFT)$}{U1r0}}

\rev{We will now show that $U_r^1(\WFT)$ admits a characterization as a conforming subspace of the Sobolev space $H^1(\text{inc})$ appearing in \eqref{eq:the-complex}. The next result will also help us}
find the local degrees of freedom of $U_r^1(\WFT)$ and $\mathring{U}_r^1(\WFT)$. 
\begin{thm}\label{thm:charU1}
    We have the following equivalent definitions of ${U}_{r}^1(\WFT)$ and $\mathring{U}_{r}^1(\WFT)$:
    \begin{align}
    \label{eqn:char1}
       & {U}_{r}^1(\WFT) = \{u \in H^1(T;\bS): u \in \cP_r(\WFT; \bS), (\bcurl u)' \in V_{r-1}^1(\WFT) \otimes \bV\}, \\
       \label{eqn:char1b}
       & \mathring{U}_{r}^1(\WFT) = \{u \in \mathring{H}^1(T;\bS): u \in \cP_r(\WFT; \bS), (\bcurl u)' \in \mathring{V}_{r-1}^1(\WFT) \otimes \bV,
      \\ \nonumber 
       & \hspace{8cm}
         {\rm inc}(u) \in \mathring{\cV}_{r-2}^2(\WFT) \otimes \bV \}.  
    \end{align}
\end{thm}

\begin{proof}
    Let the right-hand side of \eqref{eqn:char1} and \eqref{eqn:char1b} be denoted by $M_r$ and $\mathring{M}_r$, respectively. If $u \in {U}_{r}^1(\WFT)$, then $u = \sym(z)$ for some $z \in S_{r}^1(\WFT) \otimes \bV$,  so \eqref{eqn:iden5}, \eqref{eqn:iden2} and Definition \ref{def:maps} give
    \begin{equation}\label{aux617}
      (\bcurl u)'= \Xi^{-1} \bcurl u =\Xi^{-1} \bcurl z +\Grad \, \vskw(z),
    \end{equation}
    from which we conclude  $(\bcurl u)' \in V_{r-1}^1(\WFT) \otimes \bV$. This proves the inclusion 
    \begin{equation}
        {U}_{r}^1(\WFT) \subset M_r.
    \end{equation}
    Similarly, if $u \in \mathring{U}_{r}^1(\WFT)$, then \eqref{aux617} for $z \in \mathring{S}_{r}^1(\WFT) \otimes \bV$, hence we have $(\bcurl u)' \in \mathring{V}_{r-1}^1(\WFT) \otimes \bV$. Moreover, using \eqref{eqn:iden3} and the exact \rev{sequence \eqref{eqn:seq1b}}, we obtain
    \[{\rm inc}(u)=\inc(u)=\inc(z) \in \bcurl(\mathring{\Lag}_{r-1}^1(\WFT) \otimes \bV) \subset \mathring{\cV}_{r-2}^2(\WFT) \otimes \bV. \] This proves
    \begin{equation}
        \mathring{U}_{r}^1(\WFT) \subset \mathring{M}_r.
    \end{equation}

    We continue to prove the reverse inclusion of \eqref{eqn:char1}. For any $m \in M_r$, let $\sigma=\bcurl (\bcurl m)'$  which immediately implies that $\Div \sigma=0$. Moreover,  by \eqref{eqn:iden5} $\sigma= \inc(m)$ and by \eqref{eqn:iden4} $\vskw(\sigma)=0$. Hence, we have $\sigma \in V_{r-2}^2(\WFT) \otimes \bV$, and by the exact sequence  \eqref{eqn:preseq} there exists $w \in S_{r}^1(\WFT)\otimes \bV $ such that $\inc(w)=\sigma$.  Therefore, $w-m \in V_{r}^1(\WFT) \otimes \bV$ with $\inc(w-m)=0$ and hence, by the exact sequence \eqref{eqn:seq0}, there exists \rev{$v \in  {\Lag}_{r}^0(\WFT)\otimes \bV$ }such that $\Grad \, v =\Xi^{-1} \bcurl(w-m)$. Setting $z= m+ \vskw(v)$ gives  $\sym(z)=m$ and by \eqref{eqn:iden2},
    \begin{equation*}
    \bcurl z=\bcurl m+ \bcurl \mskw v=\bcurl m- \Xi \Grad \, v=\bcurl w \in  \rev{{\Lag}_{r-1}^1(\WFT)\otimes \bV}.   
    \end{equation*}
    We conclude 
    \begin{equation}
        M_r \subset {U}_{r}^1(\WFT).
    \end{equation}
        
    The reverse inclusion to prove \eqref{eqn:char1b} follows similar arguments, using the exact sequence  \eqref{eqn:preseqb} and \eqref{eqn:seq0b} in place of \eqref{eqn:preseq} and \eqref{eqn:seq0}, respectively. 
\end{proof}

\section{Local degrees of freedom for the elasticity complex on Worsey-Farin splits} \label{sec:LocalDOFs}

In this section we present degrees of freedom 
for the discrete spaces arising in the elasticity complex. We first need to introduce some notation as follows. 
Recall that $T^a$ is the set of four tetrahedra obtained by connecting the vertices of $T$ with its incenter.
For each $K \in T^a$, we denote the local Worsey-Farin splits of $K$ as $K^{wf}$, i.e., 
\[
K^{wf} = \{S \in T^{wf}: \Bar{S} \subset \Bar{K} \}. 
\]
Then, similar to the discrete functions spaces on $T^{wf}$ defined in Section \ref{subsec:localDeRhamy}, 
we define 
spaces on $K^{wf}$ by taking their restriction:
\[
\Lag_{r}^0(K^{wf}):=\{u|_K: u \in \Lag_{r}^0(T^{wf})\}; \qquad 
S_{r}^0(K^{wf}) : = \{u|_K: u \in S_{r}^0(T^{wf})\}.
\]
\begin{lemma} \label{lem:mu}
Let  $T \in \calT_h$, and let $F \in \Delta_2(T)$.  
If $p \in \Lag_r^0(T^{wf})$ with $p=0$ on $F$, then $\Grad \, p$ is continuous on $F$. In particular, 
the normal derivative $\partial_n p$ is continuous on $F$.
In addition, if $p \in S_r^0(T^{wf})$ with $p=0$ on $F$, then $\Grad \, p|_\F \in S_{r-1}^0(F^{ct}) \otimes \bV$ and in particular, $\partial_n p|_{F} \in S_{r-1}^0(\Fct)$.
\end{lemma}
\begin{proof}
Let $K \in T^a$  such that $F \in \Delta_2(K)$. Then, since $p$ vanishes on $F$, we have that   $p= \mu q$  on $K$ where $q \in \Lag_{r-1}^0(K^{wf})$ and $\mu$ is the piecewise linear polynomial in Definition \ref{def:mu}.  We  write  $\Grad \, p= \mu  \Grad \, q+ q  \Grad \, \mu$, and since $\mu$ vanishes on $F$ and $\Grad \, \mu$ is constant on $F$, we have $\Grad \, p$ is continuous on $F$. 

Furthermore, if $p \in S_r^0(T^{wf})$, then $p= \mu q$  on $K$ where $q \in S_{r-1}^0(K^{wf})$
because $\mu$ is a strictly positive polynomial on $K$. Hence by the same reasoning as the previous case, 
$\Grad \, p|_\F \in S_{r-1}^0(F^{ct}) \otimes \bV$.
\end{proof}

\subsection{Dofs of \texorpdfstring{$U^0$}{U0} space}

\begin{lemma}\label{lem:dofu0}
   A function $u \in {U}_{r+1}^0(\WFT)$, with $r \ge 3$, is fully determined by the following \dofs:
   \begin{subequations} \label{dof:U0}
       \begin{alignat}{3}
       \label{U0:dofa}
        & u(a), && \qquad a \in \Delta_0(T), && \qquad \dofcnt{12}, \\
        \label{U0:dofb}
        & {\rm grad}\, u(a), && \qquad a \in \Delta_0(T), && \qquad \dofcnt{36}, \\   
        \label{U0:dofc}
       & \int_{e}  u \cdot \kappa, && \qquad \kappa \in [\cP_{r-3}(e)]^3,~ e \in \Delta_1(T), && \qquad \dofcnt{18(r-2)}, \\
       \label{U0:dofd}
        & \int_{e}  \frac{\partial u}{\partial \bn_e^{\pm}} \cdot \kappa, && \qquad \kappa \in [\cP_{r-2}(e)]^3,~ e \in \Delta_1(T), && \qquad \dofcnt{36(r-1)}, \\
        \label{U0:dofe}
        & \int_\F \varepsilon_\F(u_{F}) \colon \varepsilon_\F(\kappa), && \qquad \kappa \in [\mathring{S}_{r+1}^0(F^{ct})]^2 ,  F \in \Delta_2(T) && \qquad \dofcnt{12r^2-36r+24}, \\
        \label{U0:doff}
        & \int_\F [\varepsilon(u)]_{\Fn} \cdot \kappa, && \qquad \kappa \in \Grad_\F \mathring{S}_{r+1}^0(F^{ct}) ,  F \in \Delta_2(T) && \qquad
        \dofcnt{6r^2-18r+12}, \\
        \label{U0:dofg}
        & \int_\F \partial_{\bn}(u \cdot \bn_\F) \kappa, && \qquad \kappa \in {\cR}_{r}^0(F^{ct}),  F \in \Delta_2(T) && \qquad
        \dofcnt{6r^2-18r+12}, \\
        \label{U0:dofh}
        & \int_\F \partial_{\bn} u_{F} \cdot \kappa, && \qquad \kappa \in [{\cR}_{r}^0(F^{ct})]^2,  F \in \Delta_2(T) && \qquad
        \dofcnt{12r^2-36r+24}, \\
        \label{U0:dofi}
        & \int_T \varepsilon(u) \colon \varepsilon(\kappa), && \qquad \kappa \in \mathring{U}_{r+1}^0(\WFT), && \qquad
        \dofcnt{6(r-1)(r-2)(r-3)},
       \end{alignat}
   \end{subequations}
   where $\frac{\partial}{\partial \bn_e^{\pm}}$ represents two normal derivatives to edge $e$ and $\{ \bn_e^{+}, \bn_e^{-}, \bt_e \}$ forms an edge-based orthonormal basis of $\mathbb{R}^3$. 
\end{lemma}

\begin{proof}
   The dimension of ${U}_{r+1}^0(\WFT)$ is $6r^3+12r+12$, which is equal to the sum of the given \dofs.  

   Let $u \in {U}_{r+1}^0(\WFT)$ such that it vanishes on the \dofs\  \eqref{dof:U0}. On each edge $e \in \Delta_1(T)$, $u|_{e}=0$ by \eqref{U0:dofa}-\eqref{U0:dofc}. Furthermore, $\Grad \, u|_e=0$ by \eqref{U0:dofb} and \eqref{U0:dofd}. Hence on any face $F \in \Delta_2(T)$, we have $u_\F \in [\mathring{S}_{r+1}^0(F^{ct})]^2$. Then with \dofs\  \eqref{U0:dofe}, $u_\F= 0$ on $F$. 
   Now with Lemma \ref{lem:mu} applied to $u_\F \in S_{r+1}^0(T^{wf}) \otimes \bV_2$, we have $\partial_{\bn} u_\F \in S_{r}^0(F^{ct}) \otimes \bV_2$. In addition, since $\Grad \, u_\F|_{\partial F} = 0$, it follows that $\partial_{\bn} u_\F \in [{\cR}_{r}^0(F^{ct})]^2$ and with \eqref{U0:dofh}, we have $\partial_{\bn} u_\F=0$. 
   

   Using the identity \eqref{more4}, we have $2 [\varepsilon(u)]_{\Fn} = \partial_{\bn} u_\F + \Grad_\F(u \cdot \bn_\F) = \Grad_\F(u \cdot \bn_\F)$. With $u \cdot \bn_\F \in S_{r+1}^0(F^{ct})$, we have in \eqref{U0:doff},  $[\varepsilon(u)]_{\Fn} = 0$ and thus $u \cdot \bn_\F = 0$ on $F$. 
   Now similar to $u_\F$, with Lemma \ref{lem:mu} applied to $u \cdot \bn_\F$, we have $\partial_{\bn} (u \cdot \bn_\F) \in {\cR}_{r}^0(F^{ct})$ and with \eqref{U0:dofg}, we have $\partial_{\bn} (u \cdot \bn_\F)=0$. 


   Since $u|_{\partial T} =0$, all the tangential derivatives of $u$ vanish. With $\partial_{\bn} (u \cdot \bn_\F)=0$ and $\partial_{\bn} u_\F=0$, we conclude that $\Grad \, u|_{\partial T} =0$. Thus $u \in \mathring{U}_{r+1}^0(\WFT)$, and \eqref{U0:dofi} shows that $u$ vanishes.
\end{proof}

\subsection{Dofs of \texorpdfstring{$U^1$}{U1} space}
Before giving the \dofs\  of the space $U^1$ we need preliminary results to see the continuity of the functions involved. In the following lemmas, we use the jump 
operator $\jmp{\cdot}$ and the set of internal edges of a split face $\Delta_1^I(F^{ct})$ 
given in Section \ref{subsec:WFconstruct}. 
The proofs of the next four results are found in the appendix.

\begin{lemma} \label{lem:symcts1}
    Let $\sigma \in V_{r}^2(\WFT) \otimes \bV$ with $\skw (\sigma) = 0$. If $\bn_\F '\sigma \ell=0$ on $\partial T$ for some $\ell \in \mathbb{R}^3$, 
   then $\sigma_{\F \ell} \in V_{{\rm div}, r}^1(F^{ct})$ on each $F \in \Delta_2(T)$.
\end{lemma}


\begin{lemma} \label{lem:curlcts}
    Let $w \in V^1_{r-1}(\WFT) \otimes \bV$ such that $w' \in V^2_{r-1}(\WFT) \otimes \bV$. If $w_{\Fn} =0$
    on some $F\in \Delta_2(T)$, then we have 
    \begin{equation}
        \rev{\jmp{\bt_s'w \bn_f}_e = 0}; \quad 
        \jmp{\bs_e' w \bs_e}_e = 0, \qquad \text{ for all } e \in \Delta_1^I(F^{ct}).
    \end{equation}
   On the other hand, if $w_\FF = 0$ on $F$, then we have
    \begin{equation}
        \label{eqn:curlcts2}
        \jmp{\bt_e'w \bn_f}_e = 0; \quad 
        \jmp{\bt_e' w \bn_F}_e = 0, \qquad \text{ for all } e \in \Delta_1^I(F^{ct}).
    \end{equation}
\end{lemma}



\begin{lemma}\label{cor:curliden}
Let $T$ be a tetrahedron, and let $\ell, m$ be two tangent vectors to a face $F\in \Delta_2(T)$ 
such that $\ell \cdot m = 0$ and $\ell \times m = \bn_\F$. 
    Let $u \in \Lag_{r}^1(\WFT) \otimes \bV$ for some $r\ge 0$. If $u_\FF = 0$ on some $F\in \Delta_2(T)$, then 
    \begin{alignat}{2}
\label{eqn:curliden1}
    \jmp{\ell'(\bcurl u )m}_e & = -\jmp{\Grad_\F(u_{\Fn} \cdot \ell) \cdot \ell}_e, & \qquad \text{ for all } e \in \Delta_1^I(F^{ct}), \\
    \label{eqn:curliden2}
    \jmp{\ell'  (\bcurl u) \ell}_e & = -\jmp{\Grad_\F(u_{\Fn} \cdot \ell) \cdot m}_e, & \qquad \text{ for all } e \in \Delta_1^I(F^{ct}).
\end{alignat}
On the other hand, if $u_\nF = 0$ on $F$, then
\begin{equation}
    \label{eqn:curliden3}
    \jmp{\bn_F'  (\bcurl u) \ell}_e  = \jmp{(\Grad_\F u_{nn}) \cdot m}_e, \qquad \text{ for all } e \in \Delta_1^I(F^{ct}).
\end{equation}
\end{lemma}

\begin{lemma} \label{lem:curlFF}
  Suppose $u \in {U}_{r}^1(\WFT)$ and $w = (\curl u)'$ are such that 
  $u_{\FF}$ and $w_{\Fn}$ vanish on a face $F\in \Delta_2(T)$. Then
  $w_{\FF} - \Grad_\F \rev{u_{n\F}^\perp}$ is continuous on $F$. Furthermore, if $u = \varepsilon(v)$ for some $v \in {U}_{r+1}^0(\WFT)$, then the following identity holds:
  \begin{equation}\label{eqn:curlFFiden}
      w_{\FF} = [(\curl \varepsilon (v))']_{\FF} =  \Grad_\F \rev{u_{n\F}^\perp} + \Grad_\F(\partial_n v_\F \times \bn_\F).
  \end{equation}
\end{lemma}

In addition to \eqref{eqn:incF} in Lemma \ref{lem:incFiden}, we need another identity to proceed with our construction. The following result is shown in \cite[Lemma 5.8]{christiansen2020discrete}. 
\begin{lemma}\label{lem:inciden}
 Let $u$ be a symmetric matrix-valued function with $[(\bcurl u)']_{\FF} \bt |_{\partial F} = 0$ , $u|_{\partial F} =0$. Let $q \in \Rig(F)$ be defined in \eqref{eqn:rigid2}. Then 
 there holds
 \begin{equation}\label{eqn:incfn}
     \int_{F} (\Inc u)_{\Fn} \cdot q  = 0.
 \end{equation}

\end{lemma}









\begin{lemma}\label{lem:dofu1}
    A function $u \in {U}_{r}^1(\WFT)$, with $r \ge 3$, is fully determined by the following vertex degrees of freedom
   \begin{subequations} \label{dof:U1}
       \begin{alignat}{3}
       \label{U1:dofa}
       & u(a), && \quad a \in \Delta_0(T), && \quad \dofcnt{24}
       \\
       \intertext{the following edge \dofs\  on  all $e \in \Delta_1(T),$}
        \label{U1:dofb}
       & \int_{e}  u \colon \kappa, && \quad \kappa \in \sym[\cP_{r-2}(e)]^{3 \times 3}, 
       && \quad \dofcnt{36(r-1)}  \\
       \label{U1:dofc}
       & \int_{e}  (\bcurl u)' \bt_e \cdot \kappa, && \quad \kappa \in [\cP_{r-1}(e)]^3, 
       && \quad \dofcnt{18r}
       \intertext{the following face \dofs\  on all $F \in \Delta_2(T)$,}
       \label{U1:dofd}
       & \int_{F} ({\rm inc~}u)_{\FF} \colon  \kappa, && \quad \kappa \in Q_{r-2}^{\perp}, 
       && \quad \dofcnt{12(r-2)}\\
       \label{U1:dofe}
        & \int_\F ({\rm inc~}u)_{nn} \kappa, && \quad \kappa \in 
        \mathring{Q}^2_{r-2}(\Fct), 
        && \quad \dofcnt{6r^2-6r-12} \\
        \label{U1:doff}
       & \int_\F ({\rm inc~}u)_{\Fn} \cdot \kappa, && \quad \kappa \in V_{{\rm div}, r-2}^1(F^{ct})/\Rig(F) ,
       && \quad \dofcnt{12r^2-24r} \\      
       \label{U1:dofg}
        & \int_{F} u_{\FF} \colon \kappa, && \quad \kappa \in \varepsilon_\F([\mathring{S}_{r+1}^0(F^{ct})]^2), 
        && \quad \dofcnt{12(r^2-3r+2)} \\
        \label{U1:dofh}
        & \int_\F  ([(\bcurl u)']_{\FF} - \Grad_\F (\rev{u_{n\F}^\perp})) 
        \colon \kappa,
        && \quad \kappa \in \Grad_\F [(\cR_{r}^0(F^{ct})]^2,  
        && \quad \dofcnt{12(r^2-3r+2)} \\
        \label{U1:dofi}
      & \int_{F} u_{\Fn} \cdot \kappa, && \quad \kappa \in \Grad_\F ([\mathring{S}_{r+1}(F^{ct})]),
      && \quad \dofcnt{6(r^2-3r+2)} \\
      \label{U1:dofj}
      & \int_{F} u_{nn} \kappa, && \quad \kappa \in \cR_{r}^0(F^{ct}), 
      && \quad \dofcnt{6(r^2-3r+2)}
      \intertext{and the following interior \dofs,}
      \label{U1:dofk}
      & \int_T {\rm ~inc}(u) \colon {\rm ~inc}(\kappa), && \quad \kappa \in \mathring{U}_{r}^1(\WFT), && \quad \dofcnt{6r^3-27r^2+21r+18} \\
      \label{U1:dofl}
      & \int_T \rev{u} \colon \varepsilon(\kappa), && \quad \kappa \in \mathring{U}_{r+1}^0(\WFT), && \quad \dofcnt{6(r-1)(r-2)(r-3)}.
       \end{alignat}
   \end{subequations}  
\end{lemma}

\begin{proof}
 The dimension of ${U}_{r}^1(\WFT)$ is $12r^3-9r^2+15r+6$, which is equal to the sum of the given \dofs. Suppose that all \dofs\  \eqref{dof:U1} vanish for a $u \in U_r^1(\WFT)$. 

{\bf Step 0:}
Using the \dofs\  (\ref{U1:dofa}, \ref{U1:dofb}) and \eqref{U1:dofc}, 
we conclude
\begin{equation}\label{eqn:U1Step0}
    u|_e = 0, \quad (\bcurl u)'\bt|_e = 0, \qquad {\rm for~} e \in \Delta_1(T).
\end{equation}

{\bf Step 1:} We show $\Inc u \in \mathring{\cV}_{r-2}^2(\WFT)\otimes \mathbb{V}$.\smallskip\\
By \eqref{id4} and \eqref{eqn:U1Step0}, 
we have \[
0=\bn_\F'(\bcurl u )' \bt = (\curl_\F u_{\FF})\bt\quad \text{on $\partial F$ for each $F\in \Delta_2(T)$.}
\]
Since $u$ is symmetric and continuous, 
by \eqref{id1}, we see that $(\Inc u)_{nn} = \incF \, u_\FF$ 
with $u_\FF \in \mathring{Q}_{{\rm inc},r}^{1,s}(\Fct) \subset \mathring{Q}_{{\rm inc},r}^{1}(\Fct)$. Thus, the complex \eqref{elaseqsvenb} in Theorem \ref{thm:2delaseq} and
the \dofs\  \eqref{U1:dofe} yield 
\begin{equation}\label{eqn:Step1Intermediate}
(\Inc u)_{nn} =0\qquad \text{on each $F\in \Delta_2(T)$.}
\end{equation}

Next, Lemma \ref{lem:symcts1} (with $\ell = \bn_\F$ and $\sigma = \Inc u$) shows $(\Inc u)_{F n}\in V^1_{{\rm div},r-2}(F^{ct})$.  Therefore using the \dofs\  \eqref{U1:doff}
and \eqref{eqn:incfn} in Lemma \ref{lem:inciden},
we conclude $(\Inc u)_{\Fn} =0$.

The identities $(\Inc u)_{nn} =0$ and $(\Inc u)_{\Fn}=0$
yield $(\Inc u)\bn_\F=0$. So, by Lemma \ref{lem:symcts1} (with $\ell = \bt_1, \bt_2$), we see that $(\Inc u)_\FF \in V_{{\rm div},r-2}^1(F^{ct})\otimes \bV_2$.  In particular, since $(\Inc u)_\FF$ is symmetric,
there holds $(\Inc u)_\FF\in Q_{r-2}^1(F^{ct})$ (cf.~\eqref{eqn:Qr1}).
Thus by the \dofs\  \eqref{U1:dofd} and the definition of $Q^{\perp}_{r-2}(F^{ct})$ in Section \ref{sec:2dela}, we have $(\Inc u)_{\FF} \in L_{r}^1(\Fct) \otimes \bV_2$. Therefore, we  conclude $\Inc u \in \mathring{\cV}_{r-2}^2(\WFT)\otimes \mathbb{V}$.


{\bf Step 2:} 
We show $(\bcurl u)' \in \mathring{V}_{r-1}^1(\WFT)\otimes \mathbb{V}$.\smallskip\\
Using \eqref{eqn:Step1Intermediate} and \eqref{id1}, we have
$0 = (\Inc u)_{nn}=\incF u_{\FF}$.
Thus by the exact sequence \eqref{elaseqsvenb} in Theorem \ref{thm:2delaseq},
there holds $u_{\FF} = \varepsilon_\F (\kappa)$ for some $\kappa\in \mathring{S}_{r+1}^0(F^{ct})\otimes \bV_2$.  We then conclude
from the \dofs\  \eqref{U1:dofg} that $u_{\FF}=0$ on each $F\in \Delta_2(F)$.
Furthermore by \eqref{id4}, $[(\bcurl u)']_{Fn} = {\rm curl}_F u_{\FF} = 0$.


Since $(\bcurl u)' \in V_{r-1}^1(\WFT)\otimes \mathbb{V}$ by Theorem \ref{thm:charU1} 
and from \eqref{eqn:U1Step0}
\[
[(\bcurl u)']_{\FF} \bt_e|_e = (\bcurl u)'\bt_e|_e = 0, \qquad \text{for all } e \in \Delta_1(T), 
\]
we have $[(\bcurl u)']_{\FF} \in \mathring{V}^1_{\curl,r-1}(F^{ct})\otimes \bV_2$ on $F \in \Delta_2(T)$. In addition, by the identity $(\Inc u)_{Fn} = {\rm curl}_F[({\rm curl}\,u)']_{\FF}$
(cf.~\eqref{id2})
and $(\Inc u)_{\Fn} =0$ derived in {\bf Step 1}, there exists $\phi \in \mathring{\Lag}^0_{r}(\Fct) \otimes \bV_2$ such that $\Grad_\F \phi = [(\bcurl u)']_{\FF}$. With Lemma \ref{lem:curlFF}, we further have $\phi - \rev{u_{n\F}^\perp} \in [\cR^0_{r}(\Fct)]^2$. 
Therefore, using the \dofs\  \eqref{U1:dofh} we conclude
\begin{equation}\label{eqn:uFn}
    [(\bcurl u)']_{\FF} = \Grad_\F  \rev{u_{n\F}^\perp}.
\end{equation}

Since with \eqref{id3}, we have
\[-\curl_\F (u_{\Fn})' = \tr_\F \bcurl u = \tr_\F (\bcurl u)' = \tr_\F (\bcurl u)'_{\FF}.\]
With \eqref{eqn:uFn} and \eqref{iden:Vperp}, we have 
\[
-\curl_\F (u_{\Fn})' = \tr_\F (\bcurl u)'_{\FF} = \DivF \rev{u_{n\F}^\perp} = \curl_\F \rev{(u_{n\F})}\rev{= \curl_\F (u_{\Fn})'},
\]
and this implies that $\curl_\F (u_{\Fn})' =0$. Since $u_{\Fn} \in \mathring{\Lag}^1_{r}(\Fct)$,
the exact sequence \eqref{2dbdryseq2} yields $u_{\Fn} \in \Grad_\F ([\mathring{S}_{r+1}(F^{ct})])$.
Therefore by \eqref{U1:dofi}, we have $u_{\Fn} = 0$. Now with \eqref{eqn:uFn} and $u_{\Fn} = 0$, 
we have $[(\bcurl u)']_{\FF} = 0$ and so $(\curl u)'\in \mathring{V}_{r-1}^1(T^{wf})\otimes \mathbb{V}$.

{\bf Step 3:} We show $u\in \mathring{H}^1(T;\mathbb{S})$.\\
From {\bf Step 2}, we already see that $u_{\FF}=0$ and $ u_{\Fn} =0$, so we 
only need to show $u_{nn} =0$. Since $(\bcurl u)' \in \mathring{V}_{r-1}^1(\WFT)\otimes \mathbb{V}$ with $\bcurl u \in V_{r-1}^2(\WFT)\otimes \mathbb{V}$ and $[(\bcurl u)']_\FF = 0$ on $F$, then by \eqref{eqn:curlcts2}, we have 
\begin{equation} 
    \jmp{\bt_e' (\bcurl u)' \bn_\F }_{e} = 0, \qquad \text{for all } e \in \Delta_1^I(\Fct).
\end{equation}
We know that $u \in \Lag_{r}^1(\WFT)$, $u_{\Fn} =0$ and by \eqref{eqn:curliden3} in Lemma \ref{cor:curliden} with $\ell = \bt_e$, $m = \bs_e$,
\[
0 = \jmp{\bt_e' (\bcurl u)' \bn_\F }_{e} = \jmp{\bn_\F' (\bcurl u) \bt_e }_{e} = \jmp{(\Grad_F u_{nn}) \cdot \bs_e}.
\]
Therefore, we have $u_{nn} \in \cR_{r}^0(F^{ct})$ and \eqref{U1:dofj} implies $u_{nn} =0$ on $F$. Thus $u|_{\partial T} =0$. 

{\bf Step 4:}\\
Using the second characterization of Theorem \ref{thm:charU1}, $u \in \mathring{U}_r^1(\WFT)$. Hence \eqref{U1:dofk} implies $\Inc u =0$ on $T$ and using the exactness of the sequence \eqref{eqn:elseqb} and the \dofs\  of \eqref{U1:dofl}, we see that $u=0$ on $T$.
\end{proof}


\subsection{Dofs of the \texorpdfstring{$U^2$}{U2} and \texorpdfstring{$U^3$}{U3} spaces}
\begin{lemma}\label{lem:dofu2}
A function $u \in {U}_{r-2}^2(\WFT)$, with $r \ge 3$, is fully determined by the following \dofs:
      \begin{subequations} \label{dof:U2}
       \begin{alignat}{3}
       \label{U2:dofa}
       & \int_{F} u_{\FF} \colon  \kappa, && \qquad \kappa \in Q_{r-2}^{\perp}, \; F \in \Delta_2(T), && \qquad \dofcnt{12(r-2)},\\
       \label{U2:dofb}
        & \int_\F u_{nn} \kappa, && \qquad \kappa \in V_{r-2}^2(F^{ct}), \; F \in \Delta_2(T), && \qquad \dofcnt{6r^2-6r}, \\
        \label{U2:dofc}
       & \int_\F u_{\nF} \cdot \kappa, && \qquad \kappa \in V_{{\rm div}, r-2}^1(F^{ct}), \; F \in \Delta_2(T), && \qquad \dofcnt{12r^2-24r+12}, \\
       \label{U2:dofd}
      & \int_T \Div u \cdot \kappa, && \qquad \kappa \in \mathring{U}_{r-3}^3(\WFT), && \qquad \dofcnt{6r^3-18r^2+12r-6}, \\
      \label{U2:dofe}
      & \int_T u \colon \kappa, && \qquad \kappa \in {\rm inc~} \mathring{U}_{r}^1(\WFT), && \qquad \dofcnt{6r^3-27r^2+21r+18}. 
       \end{alignat}
   \end{subequations}
\end{lemma}

\begin{proof}
    The dimension of ${U}_{r-2}^2(\WFT)$ is $12r^3-27r^2+15r$, which is equal to the sum of the given \dofs. 

    Let $u \in {U}_{r-2}^2(\WFT)$ such that $u$ vanishes on the \dofs\  \eqref{dof:U2}. By \dofs\  \eqref{U2:dofb}, we have $u_{nn}=0$ on each $F \in \Delta_2(T)$. By Lemma \ref{lem:symcts1} and \dofs\  \eqref{U2:dofc}, we have $u_{\nF}=0$ on each $F \in \Delta_2(T)$. Then, $u \in \mathring{V}_{r-2}^2(\WFT) \otimes \bV$. With the definition of $Q^{\perp}_{r-2}$ in Section \ref{sec:2dela} and \eqref{U2:dofa}, we have $u \in  \mathring{\cV}_{r}^2(\WFT) \otimes \bV$ and thus $u \in \mathring{U}_{r-2}^2(\WFT)$. In addition, since $\Div u \in \Div(\mathring{U}_{r-2}^2(\WFT)) \subset \mathring{U}_{r-3}^3(\WFT)$, we have $\Div u=0$ by \dofs\  \eqref{U2:dofd}. Using the exactness of \eqref{eqn:elseqb}, there exist $\kappa \in \mathring{U}_{r}^1(\WFT)$ such that $\Inc \kappa = u$. With \dofs\  \eqref{U2:dofe}, we have $u=0$, which is the desired result.
\end{proof}
\rev{
A pictorial depiction of the lowest-order space $U^2_1(\WFT)$ is given in Figure~\ref{fig:dofs}. 
We only show the \dofs\  associated to one face of the macro tetrahedron in the figure. These are the only \dofs\ that couple adjacent elements.}
\begin{figure}
          \centering
          \includegraphics[width=0.6\textwidth]{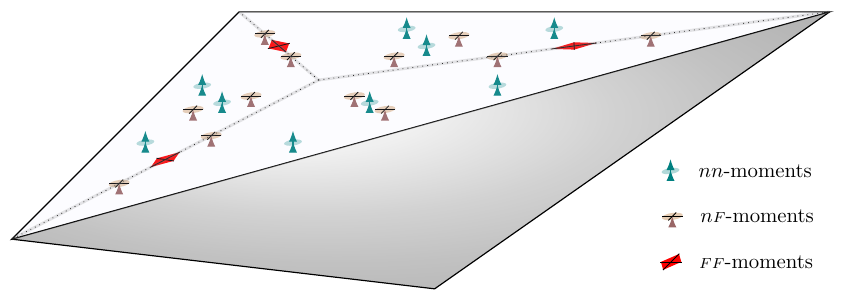}
         \caption{\rev{An illustration of coupling \dofs\  of $U^2_1(\WFT)$. Here, $\FF$-moments, $nn$-moments and $n\F$-moments
         are associated with the \dofs\  \eqref{U2:dofa},  \eqref{U2:dofb}, and \eqref{U2:dofc}, respectively. Note the absence of vertex or edge \dofs.}}
         \label{fig:dofs}
 \end{figure}


\begin{lemma}\label{lem:dofu3}
   A function $u \in {U}_{r-3}^3(\WFT)$, with $r \ge 3$, is fully determined by the following \dofs:
   \begin{subequations} \label{dof:U3}
       \begin{alignat}{3}
       \label{U3:dofa}
       & \int_T u \cdot \kappa, && \qquad \qquad \kappa \in \Rig, && \qquad \qquad
       \dofcnt{6}, \\
      \label{U3:dofb}
      & \int_T u \cdot \kappa, && \qquad \qquad \kappa \in \mathring{U}_{r-3}^3(\WFT), && \qquad \qquad \dofcnt{6r^3-18r^2+12r-6}.
       \end{alignat}
   \end{subequations}
\end{lemma}
 {
\begin{remark}
Note that \eqref{U3:dofa} is equivalent to 
\[
\int_T u \cdot \kappa, \quad \kappa \in P_\U \Rig,
\]
since by the definition of $L^2$-projection, for any $\kappa \in \Rig$, 
\[
\int_T u \cdot \kappa = \int_T u \cdot P_\U \kappa, \quad \; u \in {U}_{r-3}^3(\WFT).
\]
\end{remark}
}

\section{Commuting projections} 
In this section, we show that the degrees of freedom constructed in the previous sections induce projections which satisfy commuting properties.
\begin{thm}\label{thm:Commuting}
    Let $r \ge 3$. Let $\Pi_{r+1}^0: C^{\infty}(\Bar{T}) \otimes \bV \rightarrow U_{r+1}^0(\WFT)$ 
    be the projection defined in Lemma \ref{lem:dofu0}, 
    let $\Pi_{r}^1: C^{\infty}(\Bar{T}) \otimes \bV \rightarrow U_{r}^1(\WFT)$ be the projection defined in Lemma \ref{lem:dofu1}, 
    let $\Pi_{r-2}^2: C^{\infty}(\Bar{T}) \otimes \bV \rightarrow U_{r-2}^2(\WFT)$ be the projection defined in Lemma \ref{lem:dofu2}, 
    and let $\Pi_{r-3}^3: C^{\infty}(\Bar{T}) \otimes \bV \rightarrow U_{r-3}^3(\WFT)$ be the projection defined in Lemma \ref{lem:dofu2}.
    Then the following commuting properties are satisfied.
    \begin{subequations}
        \begin{alignat}{2}
        \label{com:epsilon}
            \varepsilon\big( \Pi_{r+1}^0 u\big) & = \Pi_{r}^1 \varepsilon(u) , && \quad \; u \in C^{\infty}(\Bar{T})\otimes \bV \\
        \label{com:inc}
            \Inc \Pi_{r}^1 v & = \Pi_{r-2}^2 \Inc v, && \quad \; v \in C^{\infty}(\Bar{T})\otimes \bS \\
        \label{com:div}
            \Div \Pi_{r-2}^2 w & = \Pi_{r-3}^3 \Div w, && \quad \; w \in C^{\infty}(\Bar{T})\otimes \bS 
        \end{alignat}
    \end{subequations}
\end{thm}

\begin{proof}
(i) Proof of \eqref{com:epsilon}: Given $u \in C^{\infty}(\Bar{T})\otimes \bV$, let $\rho = \varepsilon\big(\Pi_{r+1}^0 u\big) - \Pi_{r}^1 \varepsilon(u) \in U^1_r(\WFT)$. To prove that \eqref{com:epsilon} holds, it suffices to show that $\rho$ vanishes on the \dofs\  \eqref{dof:U1} in Lemma \ref{lem:dofu1}. Since $\Inc \circ \varepsilon = 0$, we have \dofs\  of \eqref{U1:dofd}, \eqref{U1:dofe}, \eqref{U1:doff} and \eqref{U1:dofk} applied to $\rho$ vanish. Next, with \eqref{U0:dofb}, \eqref{U0:dofe}, \eqref{U0:doff}, \eqref{U0:dofg}, 
\eqref{U0:dofi} applied to $u$, and with \eqref{U1:dofa}, \eqref{U1:dofg}, \eqref{U1:dofi}, \eqref{U1:dofj}, \eqref{U1:dofl} applied to $\varepsilon(u)$, each term respectively imply that \eqref{U1:dofa}, \eqref{U1:dofg}, \eqref{U1:dofi}, \eqref{U1:dofj}, \eqref{U1:dofl} applied to $\rho$ vanish. By the identity \eqref{eqn:curlFFiden} in Lemma \ref{lem:curlFF}, for any $\kappa \in \Grad_\F [(\cR_{r}^0(F^{ct})]^2, \text{ for all } F \in \Delta_2(T)$, we have:
\[
\int_\F ([(\bcurl \rho)']_{\FF} - \Grad_\F (\rho_{\Fn}^\perp)) \colon \kappa = \int_\F \Grad_\F(\partial_n (\Pi_{r+1}^0 u)_\F - \partial_n u_\F) \colon \kappa = 0,
\]
where the last equality holds with \eqref{U0:dofh} applied to $u$. Thus, the \dofs\  \eqref{U1:dofh} applied to $\rho$ vanish.
It only remains to prove that the \dofs\  of \eqref{U1:dofb}, \eqref{U1:dofc} applied to $\rho$ vanish. 
To show this, we need to employ the edge-based orthonormal basis
  $\{ \bn_e^+, \bn_e^-, \bt_e\}$ and write
  $\kappa \in \sym [\cP_{r-2}(e)]^{3 \times 3}$ as
  $ \kappa = \kappa_{11} \bn_e^{+} (\bn_e^{+})' + \kappa_{12} \big(\bn_e^+
  (\bn_e^-)' + \bn_e^- (\bn_e^+)'\big) + \kappa_{13} \big(\bn_e^+ \bt_e' + \bt_e
  (\bn_e^{+})'\big) +  \kappa_{22} \bn_e^-(\bn_e^-)' +
  \kappa_{23}\big(\bn_e^+(\bn_e^-)' + \bn_e^- (\bn_e^+)'\big) + \kappa_{33}
  \bt_e \bt_e'$ where $\kappa_{ij} \in \cP_{r-2}(e).  $ Then,
  \begin{alignat*}{2}
    \int_e \rho: \kappa
    & = \int_e [\varepsilon(\Pi_{r+1}^0 u)
    -\Pi_{r}^1\varepsilon(u)] : \kappa 
    = \int_e \varepsilon(\Pi_{r+1}^0 u -u) : \kappa
    & \text{by~\eqref{U1:dofb}}
    \\
    & = \int_e \Grad \,(\Pi_{r+1}^0  u - u) :  \kappa  \\
    & =  \int_e \Grad \,(\Pi_{r+1}^0  u - u)  \bt_e  \cdot ( \kappa_{13} \bn_e^+  +\kappa_{33} \bt_e) & \text{by~\eqref{U0:dofd}}\\
    & =  \int_e (\Pi_{r+1}^0  u - u) \cdot \frac{\partial}{\partial \bt_e} ( \kappa_{13} \bn_e^+  +\kappa_{33} \bt_e) & \text{integration by parts}\\
    &= 0 & \text{by~\eqref{U0:dofa} and \eqref{U0:dofc}}. 
  \end{alignat*}
  Thus
  the \dofs\  of \eqref{U1:dofb} applied to $\rho$ vanish. Next, letting
  $\kappa \in [\cP_{r-1}(e)]^3$, we note that
   \begin{alignat*}{2}
     \int_e (\curl \rho)'\bt_e \cdot \kappa=
     & \int_e \big[\curl \varepsilon(\Pi_{r+1}^0 u-u)\big]'\bt_e \cdot \kappa \quad && \text{by~\eqref{U1:dofc}} \\
    =& \frac{1}{2} \int_e  \big[\Grad \, \curl (\Pi_{r+1}^0 u-u)\big] \bt_e \cdot \kappa  \quad && \text{by~\eqref{more1}} \\
    =&   -\frac{1}{2} \int_e  \curl (\Pi_{r+1}^0 u-u)
    \cdot \partial_t\kappa \quad && \text{by~\eqref{U0:dofa} and \eqref{U0:dofb}}
    \end{alignat*}
    where in the last step, we have integrated by parts, and put
    $\partial_t \kappa = (\Grad \, \kappa) \bt_e$. The curl in the
    integrand above can be decomposed into terms involving
    $\partial_t (\Pi_{r+1}^0 u-u)$ and those involving
    $\partial_{\bn_e^\pm}(\Pi_{r+1}^0 u-u)$. The former terms can
    be integrated by parts yet again, which after
    using~\eqref{U0:dofa}, \eqref{U0:dofb} and \eqref{U0:dofc}, vanish. The
    latter terms also vanish by~\eqref{U0:dofd}, noting that
 $\partial_t \kappa$ is of degree at most $r-2$.

(ii) Proof of \eqref{com:inc}: Given $v \in C^{\infty}(\Bar{T})\otimes \bS$, let 
$\rho = \Inc \Pi_{r}^1 v - \Pi_{r-2}^2 \Inc v \in {U}_{r-2}^2(\WFT)$. To prove that 
\eqref{com:inc} holds, we need to show that $\rho$ vanishes on the \dofs\  \eqref{dof:U2} in Lemma \ref{lem:dofu2}. 
By using \eqref{U2:dofb}  on $\Inc v$, we have 
\begin{align}\label{eqn:ABC321A}
\int_\F \rho_{nn} \kappa = \int_\F [\Inc (\Pi_{r}^1 v - v)]_{nn} \kappa,\qquad \text{ for all } \kappa\in V_{r-2}^2(F^{ct}).
\end{align}
From \eqref{U1:dofe}, we have that the right-hand side of \eqref{eqn:ABC321A} 
vanishes for $\kappa \in V_{r-2}^2(F^{ct})/\cP_{1}(F)$.
With \eqref{eqn:incF} of Lemma~\ref{lem:incFiden}, we have for any $\kappa_1 \in \cP_{1}(F)$,
\begin{equation} \label{eqn:comm~incnn~2}
    \int_\F \rho_{nn} \kappa_1 = \int_{\partial F} (\curl_\F (\Pi_{r}^1 v - v)_{\FF})\bt \kappa_1  + \int_{\partial F} (\Pi_{r}^1 v - v)_{\FF} \bt \cdot (\rot_\F \kappa_1)'.
\end{equation}
By~\eqref{id4},
  $\curl_\F ( \Pi_{r}^1 v - v)_{\FF} t\, \kappa_1 = [\curl ( \Pi_{r}^1 v -
  v)']_{\Fn} t \,\kappa_1 = \curl ( \Pi_{r}^1 v - v)': \kappa_1 n t',$ so the
  first term on the right-hand side of~\eqref{eqn:comm~incnn~2} vanishes by
  \eqref{U1:dofc}.  The last term in~\eqref{eqn:comm~incnn~2} also
  vanishes because
  \begin{align*}
      \int_{\partial F} ( \Pi_{r}^1 v - v)_\FF t \cdot (\rot_\F \kappa_1)' & = \int_{\partial F} Q( \Pi_{r}^1 v - v) Q t
  \cdot (\rot_\F \kappa_1)' = \int_{\partial F} ( \Pi_{r}^1 v - v) Q t \cdot Q(\rot_\F \kappa_1)'  \\
 &=\int_{\partial F} (\Pi_{r}^1 v - v) : \sym(Q(\rot_\F \kappa_1)' t)=0,
  \end{align*}
  where we used  \eqref{U1:dofb} in the last equality.
Thus, the right-hand side of \eqref{eqn:comm~incnn~2} vanishes,
and therefore the right-hand side of \eqref{eqn:ABC321A} vanishes, i.e., the \dofs\  \eqref{U2:dofb} vanish for $\rho$. 

Next using \eqref{U2:dofc} we have
\begin{align}
\label{eqn:ABC321B}
\int_\F \rho_{\nF} \cdot \kappa = \int_\F [\Inc (\Pi_{r}^1 v - v)]_{\nF} \cdot \kappa,\qquad \text{ for all } \kappa \in V_{{\rm div},r-2}^1(F^{ct}).
\end{align}
The \dofs\  \eqref{U1:doff} imply the right-hand side of \eqref{eqn:ABC321B} vanishes  for all $\kappa \in V_{{\rm div}, r-2}^1(F^{ct})/\Rig(F)$.
Considering $\kappa\in \Rig(F)$ in \eqref{eqn:ABC321B}, we may conduct a similar argument as above,
but now  using \eqref{eqn:incfn} of Lemma~\ref{lem:inciden}, to conclude the right-hand side of \eqref{eqn:ABC321B} vanishes.
Thus, we conclude that \eqref{U2:dofc} vanishes for $\rho$.
  
  In addition, note that \eqref{U1:dofd} and \eqref{U2:dofa} imply that the \dofs\  \eqref{U2:dofa} vanish for $\rho$. Finally, the remaining \dofs\  of \eqref{U2:dofd} and \eqref{U2:dofe} applied to $\rho$ also vanish, thus leading to \eqref{com:inc}.

(iii) Proof of \eqref{com:div}: Given $w \in C^{\infty}(\Bar{T})\otimes \bS$, let $\rho = \Div \Pi_{r-2}^2 w - \Pi_{r-3}^3 \Div w \in {U}_{r-3}^3(\WFT)$. To prove that \eqref{com:div} holds, we will show that $\rho$ vanishes on the \dofs\  \eqref{dof:U3} in Lemma \ref{lem:dofu3}. Using \eqref{U2:dofd} and \eqref{U3:dofb}, we have for any $\kappa \in \mathring{U}_{r-3}^3(\WFT)$, 
\begin{align*}
    \int_{T} \rho \cdot \kappa = \int_T (\Div \Pi_{r-2}^2 w - \Div w) \cdot \kappa  = \int_T (\Div w -\Div w) \cdot \kappa = 0.
\end{align*}
For $\kappa \in \Rig$, we find 
\begin{align*}
    \int_{T} \rho \cdot \kappa & = \int_T (\Div \Pi_{r-2}^2 w - \Div w) \cdot \kappa && \text{by~} \eqref{U3:dofa}\\
    &= \int_{\partial T} (\Pi_{r-2}^2 w-w)\bn_\F \cdot \kappa\\ 
    & = \sum\limits_{F \in \Delta_2(T)} \int_{\partial F} (\Pi_{r-2}^2 w - w)_{nn} (\kappa\cdot \bn_F) -\int_{\partial F} (\Pi_{r-2}^2 w - w)_{\nF} \cdot \kappa \\
    & = 0 && \text{by~} \eqref{U2:dofb} \text{~and~} \eqref{U2:dofc},
\end{align*}
Thus, $\rho =0$, and so 
the commuting property \eqref{com:div} is satisfied.
\end{proof}

\section{Global complexes} \label{sec:global}
In this section, we construct the discrete elasticity complex globally by putting the local spaces together. Recall that $\Omega\subset \bbR^3$ is a contractible 
polyhedral domain, and $\THWFT$ is the Worsey-Farin refinement of the mesh $\calT_h$ on $\Omega$. 

We first present below the global exact de~Rham complexes on Worsey-Farin splits which are needed to construct elasticity complexes; for more details, see \cite[Section 6]{guzman2022exact}:
\begin{subequations}
 \begin{equation} \label{eqn:global~seq1}
     0  \rightarrow {\cS}_{r}^0(\THWFT) \xrightarrow{\Grad} \cL_{r-1}^1(\THWFT) \xrightarrow{\bcurl} {\mathscr{V}}_{r-2}^2(\THWFT) \xrightarrow{\Div} {V}_{r-3}^3(\THWFT) \rightarrow 0,
 \end{equation}
 \begin{equation} \label{eqn:global~seq2}
    0 \rightarrow  {\cS}_{r}^0(\THWFT) \xrightarrow{\Grad} {\cS}_{r-1}^1(\THWFT) \xrightarrow{\bcurl} \cL_{r-2}^2(\THWFT) \xrightarrow{\Div} {\mathscr{V}}_{r-3}^3(\THWFT) \rightarrow 0,
 \end{equation}
 \end{subequations}
 where the spaces involved are defined as follows:
 \begin{align*}
	\cS_r^0(\THWFT) &= \{ q \in C^1(\Omega) : q|_T \in S_r^0(\WFT), \, \text{ for all } T \in \mathcal{T}_h\}, \\
	\cS_{r-1}^1(\THWFT) &= \{ v \in [C(\Omega)]^3 : \curl v \in [C(\Omega)]^3, \, v|_T \in S_{r-1}^1(\WFT) \, \text{ for all } T \in \mathcal{T}_h\}, \\
	\mathcal{L}_{r-1}^1(\THWFT) &= \{ v \in [C(\Omega)]^3 : v|_T \in \rev{\Lag_{r-1}^1(\WFT)}, \, \text{ for all } T \in \mathcal{T}_h\}, \\
	\cL_{r-2}^2(\THWFT) &= \{ w \in [C(\Omega)]^3 : w|_T \in \rev{\Lag_{r-2}^2(\WFT)}, \, \text{ for all } T \in \mathcal{T}_h\}, \\
	\mathscr{V}_{r-2}^2(\THWFT) &= \{w \in H({\rm div} ; \Omega) : w|_T \in V_{r-2}^2(\WFT), \, \text{ for all } T \in \mathcal{T}_h, \\
	& \qquad {\theta_e(w \cdot t) =0}, \, \text{ for all } e \in \mathcal{E}(\THWFT)\}, \\
	\mathscr{V}_{r-3}^3(\THWFT) &= \{p \in L^2(\Omega) : p|_T \in V_{r-3}^3(\WFT), \, \text{ for all } T \in \mathcal{T}_h, \, \theta_e(p) = 0  \text{ and } e \in \mathcal{E}(\THWFT)\}, \\
	V_{r-3}^3(\THWFT) &= \cP_{r-3}(\THWFT),
\end{align*}
and we recall $\theta_e(\cdot)$ is defined in \eqref{eqn:thetaEDef}. \rev{Above, these spaces are defined through their continuity requirements. They can also be defined using their local \dofs\ given in \cite[Section 5.1 and Section 5.3]{guzman2022exact}. The two definitions are proven to be equivalent in \cite[Lemma 6.6 and Lemma 6.7]{guzman2022exact}. We will follow a similar approach for the elasticity complex and define the global spaces in the elasticity complex in terms of  their continuity requirements and show that the spaces are the same as those given through local \dofs.}
\rev{With the global spaces defined, the} global analogue of Theorem \ref{thm:preseq} is now given.
\begin{thm} \label{thm:global~preseq}
    The following sequence is exact for any $r \ge 3$:
    \begin{equation*}
            \begin{bmatrix}
            \cS_{r+1}^0(\THWFT) \otimes \bV \\
            \cS_{r}^0(\THWFT \otimes \bV)
            \end{bmatrix}
            \xrightarrow{[\Grad, -\mskw]} \cS_{r}^1(\THWFT) \otimes \bV \xrightarrow{\inc} \mathscr{V}_{r-2}^2(\THWFT) \otimes \bV \xrightarrow{ \left[\begin{smallmatrix}
            2\vskw \\ \Div \end{smallmatrix}\right]}
            \begin{bmatrix}
              \mathscr{V}^3_{r-2}(\THWFT) \!\otimes \!\bV \\ V_{r-3}^3(\THWFT) \otimes \bV
            \end{bmatrix}.
    \end{equation*}
Moreover, the kernel of the first operator is isomorphic to $\Rig$ and the last \rev{operator is} surjective.
\end{thm}
\begin{proof}
   The result follows from the exactness of the 
   complexes \eqref{eqn:global~seq1}-- \eqref{eqn:global~seq2}, 
   Proposition \ref{prop:exactpattern},
   and the exact same arguments in the proof of Theorem \ref{thm:preseq}.
\end{proof}

\rev{Similar to the local spaces defined in Section \ref{subsec:ela~seq},} the global spaces involved in the elasticity complex are \rev{derived} as follows:
\begin{equation}\label{eqn:Uglobal}
\begin{split}
     & U_{r+1}^0(\THWFT) = \cS_{r+1}^0(\THWFT) \otimes \bV \qquad \qquad \qquad \qquad \quad \, U_{r}^1(\THWFT) = \{\sym(u): u \in \cS_{r}^1(\THWFT) \otimes \bV\}, \\
        & U_{r-2}^2(\THWFT) = \{u \in \mathscr{V}_{r-2}^2(\THWFT) \otimes \bV: \skw \, u =0\}, \quad U_{r-3}^3(\THWFT) = V_{r-3}^3(\THWFT) \otimes \bV. 
\end{split}    
\end{equation}
\begin{thm}\label{thm:global~charU1}
    We have the following equivalent characterization of ${U}_{r}^1(\THWFT)$:
    \begin{align*}
        U_{r}^1(\THWFT) = \{ u \in H^1(\Omega;\bS): & u|_T \in U_{r}^1(\WFT), \, \text{ for all } T \in \mathcal{T}_h, \\
        & (\bcurl u)' \in {V}_{r-1}^1(\THWFT) \otimes \bV, {\rm inc}(u) \in \mathscr{V}_{r-2}^2(\THWFT) \otimes \bV \}.
    \end{align*}
\end{thm}
\begin{proof}
    This is proved similarly as the proof of Theorem \ref{thm:charU1} using Theorem \ref{thm:global~preseq} in place of Theorem \ref{thm:preseq}. 
\end{proof}
\rev{Now, we show that the global spaces
defined in \eqref{eqn:Uglobal} are equivalent to those induced
by the local \dofs\  presented in Section~\ref{sec:LocalDOFs}. To be more precise, we denote the global spaces induced by the local \dofs\   in Lemma \ref{lem:dofu0}, Lemma \ref{lem:dofu1}, Lemma \ref{dof:U2} and Lemma \ref{dof:U3} as $\tilde{U}_{r+1}^0(\THWFT)$, $\tilde{U}_{r}^1(\THWFT)$, $\tilde{U}_{r-2}^2(\THWFT)$ and $\tilde{U}_{r-3}^3(\THWFT)$, respectively. For example, 
\begin{align*}
  \tilde{U}_{r+1}^0(\THWFT) :=\{u: ~
  & u|_{T} \in U_{r+1}^0(\WFT), \text{ for all } T \in \THWFT, \text{ such that} \\ 
  & \text{the    \dofs\  \eqref{U0:dofa}-\eqref{U0:dofh}
    applied to $u$ from adjacent elements coincide}\}.
\end{align*}
The next lemma shows that such spaces are the same as those in \eqref{eqn:Uglobal}.
Its proof is similar to \cite[Lemma 6.7]{guzman2022exact}, so we will be brief.
\begin{lemma}\label{lem:local-to-global}
    The global spaces $\tilde{U}_{r+1}^0(\THWFT)$, $\tilde{U}_{r}^1(\THWFT)$, $\tilde{U}_{r-2}^2(\THWFT)$ and $\tilde{U}_{r-3}^3(\THWFT)$
    are the same as the spaces $U_{r+1}^0(\THWFT)$, $U_{r}^1(\THWFT)$, $U_{r-2}^2(\THWFT)$ and $U_{r-3}^3(\THWFT)$, respectively.
\end{lemma}}
\begin{proof}
  \rev{We only show the proof for $U_{r}^1(\THWFT)$ as the remaining cases follow by the same reasoning. To prove that $\tilde{U}_{r}^1(\THWFT) = U_{r}^1(\THWFT),$ we use the
    characterization of $U_{r}^1(\THWFT)$  in Theorem~\ref{thm:global~charU1}. Clearly,  $U_{r}^1(\THWFT) \subset \tilde{U}_{r}^1(\THWFT)$  since the continuity conditions in the characterization of Theorem~\ref{thm:global~charU1} imply that
    the \dofs\  \eqref{dof:U1}
    applied to any $u$ in $U_{r}^1(\THWFT)$ are single valued.}

\rev{
For the other direction, let  function $\chi(S)$ denote the characteristic function of a simplex $S$. 
Let $T_1$ and $T_2$ be adjacent tetrahedra in $\mathcal{T}_h$ that share a face $F$. 
Let $K_1$ and $K_2$ be two tetrahedra in the Alfeld splits $T^a_1$ and $T^a_2$, respectively, such that $K_1$ and $K_2$ share the face $F$. 
Let $K^{wf}_i$ be the triangulation of $K_i$ in $\THWFT$, where $1 \leq i \leq 2$. 
Let $u_1 \in U_{r}^1(\WFT_1)$ and $u_2 \in U_{r}^1(\WFT_2)$ such that $u_1$ and $u_2$ have the same \dof\ values \eqref{U1:dofa}-\eqref{U1:dofj} associated with the common vertices, common edges and the triangulation $\Fct$.  Note that the natural extension of $u_1$ (resp., $u_2$) from $K_1^{wf}$ (resp., $K_2^{wf}$) to all of $K_1^{wf} \cup K_2^{wf} $ maintains its original smoothness properties across the interior faces of $K_2^{wf}$ (resp., $K_1^{wf}$).
Thus, by applying the unisolvency argument
in the proof of Lemma \ref{lem:dofu1} verbatim to $w:=u_1-u_2$,
we conclude that $w=0$, $(\curl \, w)'_{\FF}=0$, $(\curl \, w)'_{\Fn}=0$, $(\Inc \, w)n_{\F} = 0$ and $(\Inc \, w)_{\FF} = 0$ on $F$.
Therefore, $u:=u_1 \chi(T_1)+u_2 \chi(T_2) \in U^1_{r}(\WFT_1 \cup \WFT_2)$, and we conclude the reverse inclusion $\tilde{U}_{r}^1(\THWFT) \subset U_{r}^1(\THWFT).$
}
\end{proof}

Then we have the global complex summarized in the following theorem. Its proof follows along the same lines
 as Theorem \ref{thm:elseq}, with Theorem \ref{thm:global~preseq} in place of Theorem \ref{thm:preseq}.
\begin{thm} \label{thm:global~elseq}
    The following sequence of global finite element spaces
     \begin{equation}\label{eqn:global~elseq}
         0 \rightarrow \Rig  \xrightarrow{\subset} {U}_{r+1}^0(\THWFT) \xrightarrow{\varepsilon} {U}_{r}^1(\THWFT) \xrightarrow{\Inc} {U}_{r-2}^2(\THWFT) \xrightarrow{\Div} {U}_{r-3}^3(\THWFT) \rightarrow 0
    \end{equation}
    is a discrete elasticity complex and is exact for $r \ge 3$.
\end{thm}

\section{Conclusions}\label{sec:Conclude}
This paper constructed
both local and global finite element elasticity complexes with respect 
to three-dimensional Worsey-Farin splits. 
A notable
feature of the discrete spaces is the lack of 
extrinsic supersmoothess and
accompanying \dofs\  at vertices in the triangulation.  
For example,
the $H({\rm div}, \mathbb{S})$-conforming space does not involve
vertex or edge \dofs\  and is therefore conducive for hybridization.
\rev{The efficient implementation of these elements with hybridization,
with an emphasis on the lowest-order pair, is a subject of future work.
Our  results suggest} 
that the last two pairs in the sequence \eqref{eqn:global~elseq} are suitable
to construct mixed finite element methods for three-dimensional elasticity.
However, due to the assumed regularity in Theorem \ref{thm:Commuting},
the result does not automatically yield an inf-sup stable pair.
Further study of commuting projections for 
the pair $U_{r-2}^2(\THWFT)\times U_{r-3}^3(\THWFT)$ is required 
to prove inf-sup stability.


 \appendix
\section{Proof of Theorem \ref{thm:2delaseq}}\label{apdix:sec3}
We require a few intermediate results
to prove Theorem \ref{thm:2delaseq}.
First, we state a corollary of 
Theorem \ref{2dseqs}.
\begin{cor} \label{cor:gradcurl~sven}
 Let $r \ge 1$. The following sequence is exact.
 \begin{equation}
     0\
{\xrightarrow{\hspace*{0.5cm}}}\
\mathring{S}_{r}^0(\Fct) \otimes \bV_2 \
\stackrel{{\rm grad}_\F}{\xrightarrow{\hspace*{0.5cm}}}\
\mathring{Q}_{{\rm inc}, r-1}^1(\Fct)\
\stackrel{\curl_\F}{\xrightarrow{\hspace*{0.5cm}}}\
\mathring{V}_{\curl,r-2}^1(\Fct) \cap \big(\mathring{V}_{r-2}^2(\Fct) \otimes \bV_2\big)\
\xrightarrow{\hspace*{0.5cm}}\
 0. \label{eqn:gradcurl~sven}
 \end{equation}     
 \end{cor}
 \begin{proof}
 This directly follows from the exactness of the sequence 
     \eqref{2dbdryseq2}.
 \end{proof}

\begin{lemma}\label{lem:2dpreela}
    The following sequences are exact for $r \ge 2$:
    \begin{equation} \label{eqn:2dpreelasvenb}
         \begin{bmatrix}
        \mathring{{S}}_{r+1}^0(\Fct) \otimes \bV_2 \\
        \mathring{\Lag}_{r}^0(\Fct)
    \end{bmatrix} \xrightarrow{\begin{bmatrix}
        \Grad_\F & {\rm skew}
    \end{bmatrix}}  \mathring{{Q}}_{{\rm inc}, r}^1(\Fct) \otimes \bV_2 \xrightarrow{\incF } {V}_{r-2}^2(\Fct)  \xrightarrow{\begin{bmatrix}
        \int_\F^{\perp} \\ \int_\F
    \end{bmatrix}}
       \begin{bmatrix}
        \bV_2 \\
        \mathbb{R}
    \end{bmatrix},
    \end{equation}
    \begin{equation} \label{eqn:2dpreelaairy}
         \begin{bmatrix}
         \mathbb{R} \\
        \bV_2
    \end{bmatrix} 
    \xrightarrow{\begin{bmatrix}
        \subset & x^\perp \cdot
    \end{bmatrix}}
    {{S}}_{r+1}^0(\Fct)
    \xrightarrow{\airy}
          {{V}}_{{\rm div}, r-1}^1(\Fct) \otimes \bV_2  \xrightarrow{\begin{bmatrix}
        {\rm skew} \\ \DivF
    \end{bmatrix}}
        \begin{bmatrix}
        {{V}}_{r-1}^2(\Fct)  \\
        {{V}}_{r-2}^2(\Fct) \otimes \bV_2
    \end{bmatrix} .
    \end{equation}
\end{lemma}
Here, $\int_\F^{\perp} u := \int_\F x^{\perp} u ~dx$ with $x^\perp$ defined in Definition \ref{def:Vperp}.
\begin{proof}

Using \eqref{2didenela3} and the identity
$\int_\F^{\perp} \, \curl_\F \, u = \int_\F \, \tau \, u$ for any $u \in \mathring{{V}}_{\curl,r-1}^1(\Fct)$, 
we find that the following sequence commutes:
\begin{equation}\label{eqn:2ddiagramsvenb}
    \begin{tikzcd}
 \mathring{{S}}_{r+1}^0(\Fct) \otimes \bV_2 \arrow{r}{\Grad_\F} & \mathring{{Q}}_{{\rm inc}, r}^1(\Fct) \otimes \bV_2 \arrow{r}{\curl_\F} & \mathring{{V}}_{\curl,r-1}^1(\Fct) \arrow{r}{\int_\F} & \bV_2 \\ 
\mathring{{\Lag}}_{r}^0(\Fct) \arrow{r}{\Grad_\F}  \arrow[ur, "{{\rm skew}}"] & \mathring{{V}}_{\curl,r-1}^1(\Fct) \arrow{r}{\curl_\F} \arrow[ur, "\tau"] & {V}_{r-2}^2(\Fct) \arrow{r}{\int_\F} \arrow[ur, "{\int_\F^{\perp}}"]& \mathbb{R},
\end{tikzcd}
\end{equation}
Moreover, 
the transpose operator $\tau$ from 
$\mathring{{V}}_{\curl,r-1}^1(\Fct)$ to $\mathring{{V}}_{\curl,r-1}^1(\Fct)$ is a bijection,
and the top and bottom sequences in \eqref{eqn:2ddiagramsvenb}
are exact by Corollary \ref{cor:gradcurl~sven}
and Theorem \ref{2dseqs}, respectively.
Using the identity $\incF = \curl_\F \tau \curl_\F$
and Proposition \ref{prop:exactpattern},
we conclude that \ref{eqn:2dpreelasvenb} is exact.

Likewise, using the identity 
$\DivF \, \tau \, u = {\rm skew} \, \rot_\F \, u$ for any $u \in ({{\Lag}}_{r}^0(\Fct) \otimes \bV_2 )$ and $\rot_\F x^\perp =\tau$,
we find that the following sequence commutes:
\begin{equation}\label{eqn:2ddiagramairy}
    \begin{tikzcd}
\mathbb{R} \arrow{r}{\subset} &  {{S}}_{r+1}^0(\Fct)  \arrow{r}{\rot_\F} & {{\Lag}}_{r}^1(\Fct)  \arrow{r}{\DivF} & {V}_{r-1}^2(\Fct)  \\ 
 \bV_2 \arrow{r}{\subset} \arrow[ur, "x^\perp \cdot"] &  {{\Lag}}_{r}^0(\Fct) \otimes \bV_2 \arrow{r}{\rot_\F} \arrow[ur, "\tau"] & {{V}}_{{\rm div},r-1}^1(\Fct) \otimes \bV_2 \arrow{r}{\DivF} \arrow[ur, "{\rm skew}"] & {V}_{r-2}^2(\Fct) \otimes \bV_2.
\end{tikzcd}
\end{equation}
The top and bottom sequences in \eqref{eqn:2ddiagramairy}
are exact by Corollary \ref{cor:rotdiv}.
We then find that \eqref{eqn:2dpreelaairy}
is exact by Proposition \ref{prop:exactpattern},
using the identity $\airy = \rot_F \tau\rot_F$.

\end{proof}

Now we are ready to prove Theorem \ref{thm:2delaseq}:
\begin{proof}
    (i) Proof of \eqref{elaseqsvenb}: 
    from the definitions of the discrete spaces and operators, we see that \eqref{elaseqsvenb} is a complex, so we only need to show exactness. 
    
    Let $v \in \mathring{Q}^2_{r-2} (\Fct)$. Then since $v\perp \mathcal{P}_1(F)$, we have $\int_\F v =0$ and 
    $\int_\F^\perp v = 0$. By the exactness of \eqref{eqn:2dpreelasvenb}, 
    there exists $u \in  \mathring{Q}_{{\rm inc},r}^{1}(\Fct)$ such that $\incF \, u = v$. But by \eqref{2didenela2}, we have $\incF\, \sym \, u = \incF \, u =v$. Thus we found a function $w = \sym \, u \in \mathring{Q}_{{\rm inc},r}^{1,s}(\Fct)$ such that $\incF \, w = v$.   
    
    Next, let $u \in \mathring{Q}_{{\rm inc},r}^{1,s}(\Fct)$ with $\incF \, u = 0$. Then $u = \sym(z)$ for some $z \in \mathring{Q}_{{\rm inc},r}^{1}(\Fct)$ and $\incF \, z = 0$ due to \eqref{2didenela2}. By exactness of \eqref{eqn:2dpreelasvenb}, we have $z = \Grad_\F w + {\rm skew} \, s$ for some $w \in  \mathring{{S}}_{r+1}^0(\Fct) \otimes \bV_2$ and $s \in  \mathring{{\Lag}}_{r}^0(\Fct)$. Then $u = \sym (z) = \varepsilon_\F(w) - \sym ({\rm skew} \, s) = \varepsilon_\F(w)$.

    (ii) Proof of \eqref{elaseqairy}: again, it is easy to see that \eqref{elaseqairy} is a complex, so we only need to show exactness. 
    
    Let $v \in V^2_{r-3} (\Fct) \otimes \bV_2$. Then by the exactness of \eqref{eqn:2dpreelaairy}, we have $u \in  {{V}}_{{\rm div}, r-2}^1(\Fct) \otimes \bV_2$ such that $\DivF \, u = v$ and ${\rm skew} \, u =0$ and thus making $u \in Q^1_{r-2}(\Fct)$.   
   
   Next, let $u \in Q^1_{r-2}(\Fct)$ with $\DivF \, u = 0$. 
   Then again using \eqref{eqn:2dpreelaairy} and  ${\rm skew} \, u =0$, 
   there exists $z \in {{S}}_{r}^0(\Fct)$ such that $\airy \, z = u$. 
   
   Finally, for any $u \in {{S}}_{r}^0(\Fct)$ with $\airy \, u =0$, we have $u = w + x^\perp \cdot s$ for some $w \in  \mathbb{R}$, $s \in  \bV_2$, and 
   $x$ a point on the face $F$. Therefore, $u \in \cP_1(F)$. 

\end{proof}

\section{Proof of Lemma \ref{lem:projRig}}
\begin{proof}

  We first show that $\dim P_\U \Rig=\dim \Rig=6$. This follows if we
  show that the kernel of $P_U$ is empty. Let $v \in \Rig$ and assume
  that $P_\U v=0$. Then, by the definition of $P_\U$ and the fact that
  $v$ is a linear function, we must have that $v$ vanishes on the
  barycenter of each $K \in \WFT$. This implies that $v \equiv 0$ if
  there are three such barycenters that are not collinear. To see that
  there are such barycenters, recall that the barycenter of
  $K \in \WFT$ is the average of the four vertices of $K$.  Hence the
  line connecting barycenters of two adjacent $K_\pm \in \WFT$ is
  parallel to the line connecting the two vertices opposite to the
  common face $F = \partial K_+ \cap \partial K_-$.  Thus taking, for
  example, three subtetrahedra in $\WFT$ with a face contained in a
  common $F \in \Delta_2(\mathcal{T}_h)$, we see that their
  barycenters cannot be collinear, since no three of their vertices
  are collinear.


We now prove \eqref{U03}.  Since $\dim \Rig = 6$ and by the definition of $\mathring{U}_{0}^3(\WFT)$, we have
\[
\dim \mathring{U}_{0}^3(\WFT) \ge \dim {U}_{0}^3(\WFT) - \dim \Rig =36-6 =30.
\]
We use that 
\begin{equation*}
{U}_{0}^3(\WFT)= \mathring{U}_{0}^3(\WFT) \oplus [\mathring{U}_{0}^3(\WFT)]^\perp,
\end{equation*}
and obtain  $\dim [\mathring{U}_{0}^3(\WFT)]^\perp \le 6$. 
However, one can easily show that  
$P_\U\Rig \subset [\mathring{U}_{0}^3(\WFT)]^\perp$ 
which implies $\dim [\mathring{U}_{0}^3(\WFT)]^\perp=6$ and  $P_\U\Rig =[\mathring{U}_{0}^3(\WFT)]^\perp$.

\end{proof}
 
\section{Proof of Lemma \ref{lem:symcts1}}
\begin{proof}
   Fix $F\in \Delta_2(T)$, and let $e \in \Delta_1^I(F^{ct})$
   be an internal edge in the induced Clough-Tocher split of $F$. 
   Let $f$ be the corresponding internal face of $\WFT$ with $e$ as an edge, and let $\bn_f$ is a unit-normal to $f$. 
   We further set $t_e$ to be a unit tangent vector to $e$ and $s_e = n_F\times t_e$ to be
   a unit tangent vector of $F$ orthogonal to $t_e$.

     Since $\bn_f \cdot \bt_e = 0$, we have $\bn_f = (\bn_f \cdot \bn_\F) \bn_\F + (\bn_f \cdot \bs_e)\bs_e$. 
    Since $\sigma \in V_{r}^2(\WFT) \otimes \bV$, we have $\sigma \bn_f$ is single-valued on $e$ and hence, by symmetry of $\sigma$, $(\sigma \ell) \cdot \bn_f$ is single-valued on $e$. Therefore, on $e$, with $(\sigma \ell) \cdot \bn_\F = \bn_\F '\sigma \ell = 0$, we have 
    $(\sigma \ell)\cdot \bn_f = (\bn_f \cdot \bs_e)(\sigma \ell)\cdot \bs_e$ and so $\jmp{\sigma_{\F \ell} \cdot \bs_e }_{e}=\jmp{(\sigma \ell)\cdot \bs_e }_{e}=0$ for any $e \in \Delta_1^I(F^{ct})$. Therefore, $\sigma_{\F \ell} \in V_{{\rm div}, r}^1(\Fct)$ on each $F \in \Delta_2(T)$.
\end{proof}

\section{Proof of Lemma \ref{lem:curlcts}}
\begin{proof}
Since $w \in V^1_{r-1}(\WFT) \otimes \bV$ and $w' \in V^2_{r-1}(\WFT) \otimes \bV$, then $\bn_f' w$, $w \bt_e$ and $w \bt_s$ are  continuous cross $e$ on $F$:
\begin{equation}\label{eqn:curlcts}
    \jmp{\bn_f' w}_{e} = 0, \quad \jmp{w \bt_e}_{e} = 0, \quad \jmp{w \bt_s}_{e} = 0.
\end{equation}
Let $\bs_e = \alpha_1 \bn_f + \beta_1 \bt_s$, $\bn_\F = \alpha_2 \bn_f + \beta_2 \bt_s$,
and note $\alpha_1\neq 0$ and $\beta_2 \neq 0$.

Since  \rev{$\bn_\F'wQ|_{F} = 0$}, for any $e \in \Delta_1^I(F^{ct})$, 
\begin{align*}
    0 = \jmp{\bn_\F' w \bs_e }_{e} & = \jmp{(\alpha_2 \bn_f' + \beta_2 \bt_s') w (\alpha_1 \bn_f + \beta_1 \bt_s) }_e \\
    & = \alpha_1 \alpha_2 \jmp{\bn_f'w \bn_f}_{e}+\alpha_2\beta_1 \jmp{\bn_f'w\bt_s}_e+\alpha_1\beta_2 \jmp{\bt_s'w \bn_f}_e + \beta_2\beta_1\jmp{\bt_s' w \bt_s}_e\\
    &{=\alpha_1\beta_2 \jmp{\bt_s'w \bn_f }_e.}
\end{align*}
Thus, we have
\begin{equation*}
    \jmp{\bt_s'w \bn_f}_e = 0,
\end{equation*}
and therefore
\begin{equation*}
    \jmp{\bs_e' w \bs_e}_e = \alpha_1^2\jmp{\bn_f'w \bn_f}_e+\alpha_1\beta_1\jmp{\bn_f'w\bt_s}_e+\alpha_1\beta_1\jmp{\bt_s'w \bn_f}_e + \beta_1^2 \jmp{\bt_s' w \bt_s}_e = 0.
\end{equation*} 

We have $\jmp{\bt_e' w \bn_f}_{e}=0$ since $w_{\FF} = 0$ and
\begin{align*}
    0 = \jmp{\bt_e' w \bs_e}_{e} = \jmp{\bt_e' w (\alpha_1 \bn_f + \beta_1 \bt_s) }_e = \alpha_1 \jmp{\bt_e'w \bn_f}_{e}+\beta_1 \jmp{\bt_e'w\bt_s}_e {= \alpha_1  \jmp{\bt_e'w \bn_f}_{e}},
\end{align*}
where we use \eqref{eqn:curlcts}. This implies that 
\begin{equation*}
    \jmp{\bt_e' w \bn_\F }_{e} = 0
\end{equation*}
since 
\[
\jmp{\bt_e' w \bn_\F }_{e} = \jmp{\bt_e' w (\alpha_2 \bn_f + \beta_2 \bt_s) }_{e} = 0. 
\]
\end{proof}

 \section{Proof of Lemma \ref{cor:curliden}}


 \begin{proof}
Write $\ell = a_1 \bt_1 + a_2 \bt_2$, $m = a_1\bt_2 - a_2 \bt_1$, where 
$\bt_1, \bt_2$ are tangential basis defined in Section \ref{subsec:Matrix}.
We also set $\bt_3 = \bn_\F$, and write $u = \sum\limits_{i,j = 1}^3 u_{ij} \bt_i \bt_j'$.
We then have the following identities 
for the components of $\bcurl u$ ($s\in \{1,2,3\})$:
    \begin{equation}\label{eqn:curl~represent}
    \begin{split}
        \bt_s' (\bcurl u) \bt_1 &= \partial_{\bt_2} u_{s3}-\partial_{\bt_3} u_{s2},\\
        \bt_s' (\bcurl u) \bt_2 &= \partial_{\bt_3} u_{s1}-\partial_{\bt_1} u_{s3},\\
         \bt_s' (\bcurl u) \bt_3 &= \partial_{\bt_1} u_{s2}-\partial_{\bt_2} u_{s1}.
         \end{split}
    \end{equation}

We then compute    
    \begin{equation}
    \label{eqn:Oner}
    \begin{split}
        \ell'(\bcurl u )m & =  (a_1 \bt_1 + a_2 \bt_2)' (\bcurl u )(a_1\bt_2 - a_2 \bt_1) \\ 
        & = (a_1)^2(\partial_{\bt_3} u_{11}-\partial_{\bt_1} u_{13})-(a_2)^2(\partial_{\bt_2} u_{23}-\partial_{\bt_3} u_{22}) \\
        & \qquad +a_1a_2(\partial_{\bt_3} u_{21}-\partial_{\bt_1} u_{23} - \partial_{\bt_2} u_{13}+\partial_{\bt_3} u_{12}) \\
        & = \partial_{\bt_3}((a_1)^2u_{11}+(a_2)^2u_{22}+a_1a_2(u_{21}+u_{12})) \\
        & \qquad - a_1(a_2\partial_{\bt_2} u_{13}+a_1\partial_{\bt_1} u_{13})-a_2(a_1\partial_{\bt_1} u_{23}+a_2\partial_{\bt_2} u_{23}) \\
        & = \partial_{\bt_3}(\ell'u_{\FF}\ell) - a_1 \partial_{\ell} u_{13} - a_2 \partial_{\ell} u_{23} \\
        & = \partial_n (\ell'u_{\FF}\ell)-\partial_l(u_{\Fn} \cdot \ell) = \partial_n (\ell'u_{\FF}\ell)-\Grad_\F (u_{\Fn} \cdot \ell) \cdot \ell.
        \end{split}
    \end{equation}
    Similarly, by using \eqref{eqn:curl~represent}, we have 
    \begin{equation}
    \begin{split}
    \label{eqn:Twoer}
        \ell'(\bcurl u ) \ell & =  (a_1 \bt_1 + a_2 \bt_2)' (\bcurl u )(a_1\bt_1 + a_2 \bt_2) \\ 
        & = (a_1)^2(\partial_{\bt_2} u_{13}-\partial_{\bt_3} u_{12})+(a_2)^2(\partial_{\bt_3} u_{21}-\partial_{\bt_1} u_{23}) \\
        & \qquad +a_1a_2(\partial_{\bt_2} u_{23}-\partial_{\bt_3} u_{22} + \partial_{\bt_3} u_{11}-\partial_{\bt_1} u_{13}) \\
        & = \partial_{\bt_3}(-(a_1)^2u_{12}+(a_2)^2u_{21}+a_1a_2(u_{11}-u_{22})) \\
        & \qquad - a_1(-a_1\partial_{\bt_2} u_{13}+a_2\partial_{\bt_1} u_{13})+a_2(a_1\partial_{\bt_2} u_{23}-a_2\partial_{\bt_1} u_{23}) \\
        & = -\partial_{\bt_3}(m'u_{\FF}\ell) + a_1 \partial_{m} u_{13} + a_2 \partial_{m} u_{23} \\
        & = -\partial_n (m'u_{\FF}\ell)+\partial_m(u_{\Fn} \cdot \ell) = -\partial_n (m'u_{\FF}\ell)+\Grad_\F (u_{\Fn} \cdot \ell) \cdot m.
    \end{split}
    \end{equation}
    Finally, again by using \eqref{eqn:curl~represent}, we have 
    \begin{equation}
    \label{eqn:Threeer}
    \begin{split}
        \bn_F'(\bcurl u ) \ell & =  \bt_3' (\bcurl u )(a_1\bt_1 + a_2 \bt_2) \\ 
        & = (a_1\partial_{\bt_2} u_{33} - a_2\partial_{\bt_1} u_{33}) + \partial_{\bt_3} (a_2u_{31}-a_1a_{32}) \\
        & = \partial_m u_{33} - \partial_n(u_\nF \cdot m) = (\Grad_F u_{33}) \cdot m - \partial_n(u_\nF \cdot m).
    \end{split}
    \end{equation}
Lemma \ref{cor:curliden} now follows from \eqref{eqn:Oner}--\eqref{eqn:Threeer}
and the first case in Lemma \ref{lem:mu}.    
\end{proof}

\section{Proof of Lemma \ref{lem:curlFF}}
\begin{proof}
    (i) {\bf Continuity:} we show the continuity of $w_{\FF} - \Grad_\F \rev{u_{n\F}^\perp}$. Recall the notation from Section \ref{subsec:WFconstruct}. 
    Since $w \in V^1_{r-1}(\WFT) \otimes \bV$ {(by Theorem \ref{thm:charU1})}, 
for any $e\in \Delta_1^I(F^{ct})$    we have $\jmp{w_{\FF} \bt_e}_e = 0$ due to $\jmp{w \bt_e}_e = 0$. Consequently, 
    because $u$ is continuous, we have 
    \begin{equation}\label{eqn:curlt}
    \jmp{(w_{\FF} - \Grad_\F \rev{u_{n\F}^\perp}) \bt_e }_{e} = 0.
\end{equation}
Now to prove the continuity of $w_{\FF} - \Grad_\F \rev{u_{n\F}^\perp}$ on $F$, it suffices to prove
$\jmp{(w_{\FF} - \Grad_\F \rev{u_{n\F}^\perp}) \bs_e }_{e} = 0$ for all $e\in \Delta^I_1(F^{ct})$.
Using $w' \in V^2_{r-1}(\WFT) \otimes \bV$ and $w_{Fn} =0$,
by Lemma \ref{lem:curlcts} we have
    \begin{equation} \label{eqn:curlss}
        \jmp{\bs_e' w_{\FF} \bs_e }_{e}=\jmp{\bs_e' w \bs_e}_e=0.
    \end{equation}
    
    


Next we show that $\jmp{\bs_e' \Grad_\F(\rev{u_{n\F}^\perp}) \bs_e }_{e}=0$ and $\jmp{\bt_e' (w_{\FF} - \rev{u_{n\F}^\perp}) \bs_e }_{e}=0$. Since $u \in \Lag_{r}^1(\WFT) \otimes \bV$ and $u_{\FF} =0$ on $F$, we have
\begin{align}\label{eqn:uFns1}
   \jmp{\bs_e' \Grad_\F(\rev{u_{n\F}^\perp}) \bs_e }_{e} & =\jmp{\Grad_\F(\rev{u_{n\F}^\perp} \cdot \bs_e) \cdot \bs_e }_{e}=\jmp{\Grad_\F(u_{\Fn}^\perp \cdot \bs_e) \cdot \bs_e }_{e} \nonumber \\
   & =\jmp{\Grad_\F(u_{\Fn} \cdot \bt_e) \cdot \bs_e }_{e} = -\jmp{\bt_e' (\bcurl u)' \bt_e }_{e} = 0, 
\end{align}
where the \rev{third} equality comes from \eqref{iden:Vperp} and the \rev{fourth} equality uses \eqref{eqn:curliden2} in Lemma \ref{cor:curliden} with $\ell = \bt_e$ and $m =\bs_e$. 
Similarly by \eqref{eqn:curliden1} in Lemma \ref{cor:curliden} with $\ell = \bs_e$, $m = - \bt_e$ and \eqref{iden:Vperp}, we have $\jmp{\bt_e' \Grad_\F(\rev{u_{n\F}^\perp}) \bs_e }_{e}= \jmp{\bt_e' (\bcurl u)' \bs_e }_{e}$. 
Therefore, we have
\begin{equation}\label{eqn:uFns2}
   \jmp{\bt_e' (w_{\FF} - \Grad_F \rev{u_{n\F}^\perp}) \bs_e }_{e}=\jmp{\bt_e' [(\bcurl u)']_{\FF} \bs_e}_e-\jmp{\bt_e' (\bcurl u)' \bs_e }_{e} = 0.
\end{equation}
Combining \eqref{eqn:curlt}, \eqref{eqn:curlss}, \eqref{eqn:uFns1} and \eqref{eqn:uFns2}, we conclude that $w_{\FF} - \Grad_\F \rev{u_{n\F}^\perp}$ is continuous on $F$.

(ii) {\bf Proof of \eqref{eqn:curlFFiden}:} With \eqref{more2},  \eqref{more3} and \eqref{eq:10}, we have 
\[
2w_{\FF} = \Grad_\F (\Grad_\F (v \cdot \bn_\F) \times \bn_\F  - (\partial_n v_\F)\times \bn_\F).
\]
Then with \eqref{more4}, \eqref{more5} and \eqref{iden:Vperp}, we obtain
\[
2\Grad_\F u_{\nF}^\perp = 2 \Grad_\F [(\varepsilon(v))_{\nF}]^\perp =
\Grad_\F(\Grad_\F(v \cdot \bn_\F) \times \bn_\F + (\partial_n v_\F)\times \bn_\F).
\]
Therefore, by computing the difference of the above two equations, we conclude 
that
$
w_{\FF} - \Grad_\F u_{\nF}^\perp = \Grad_\F(\partial_n v_\F \times \bn_\F).
$
\end{proof}

\bibliographystyle{siam}
\bibliography{ref}

\end{document}